\documentclass[12pt, reqno]{amsart}

\usepackage[hmargin=0.8in,height=8.6in]{geometry}
\usepackage{amssymb,amsthm, times}
\usepackage{delarray,verbatim}
\usepackage{ifpdf}
\ifpdf
\usepackage[pdftex]{graphicx}
 \else
\usepackage[dvips]{graphicx}
 \fi
\usepackage{bm}

\usepackage{ifpdf}
\usepackage{color}
\definecolor{webgreen}{rgb}{0,.5,0}
\definecolor{webbrown}{rgb}{.8,0,0}
\definecolor{emphcolor}{rgb}{0.95,0.95,0.95}

\usepackage{hyperref}
\hypersetup{%
          colorlinks=true,
          linkcolor=webbrown,
          filecolor=webbrown,
          citecolor=webgreen,
          breaklinks=true}
\ifpdf \hypersetup{pdftex,
            pdfstartview=FitH, 
            bookmarksopen=true,
            bookmarksnumbered=true
} \else \hypersetup{dvips} \fi

\renewcommand{\S}{\mathcal{S}}
\newcommand{\ST}{\widetilde{\S}}

\numberwithin{equation}{section}

\newtheorem{theorem}{Theorem}[section]
\newtheorem{proposition}{Proposition}[section]
\newtheorem{corollary}{Corollary}[section]
\newtheorem{remark}{Remark}[section]
\newtheorem{lemma}{Lemma}[section]

\newtheorem{assump}{Assumption}[section]

\numberwithin{remark}{section} \numberwithin{proposition}{section}
\numberwithin{corollary}{section}
\newcommand {\R}{\mathbb{R}}

\newcommand {\F}{\mathcal{F}}

\newcommand {\N}{\mathbb{N}}
\newcommand {\p}{\mathbb{P}}

\newcommand {\E}{\mathbb{E}}

\newcommand {\II}{\mathcal{I}}

\newcommand{\diff}{{\rm d}}

\newcommand{\lev}{L\'{e}vy }

\newcommand{\lapinv}{\Phi_r}

\title[Contraction options in spectrally negative L\'{e}vy models]{Contraction options and optimal multiple-stopping \\ in spectrally negative L\'{e}vy models}
\thanks{This version: \today. The author thanks the anonymous referee for his/her thorough reviews and insightful comments that help improve the presentation of the paper.   K.\ Yamazaki is in part supported by MEXT KAKENHI grant numbers  22710143 and 26800092, JSPS KAKENHI grant number 23310103, the Inamori foundation research grant, and the Kansai University subsidy for supporting young scholars 2014.}
\thanks{$^*$\, Department of Mathematics,
Faculty of Engineering Science, Kansai University, 3-3-35 Yamate-cho, Suita-shi, Osaka 564-8680, Japan. Email: \mbox{{\em
kyamazak@kansai-u.ac.jp}}.  Tel: +81-6-6368-1527.}
\author[K. Yamazaki]{Kazutoshi Yamazaki$^*$}
\date{}
\begin{document}
\begin{abstract}
This paper studies  the optimal multiple-stopping problem arising in the context of the timing option to withdraw from a project in stages.  The profits are driven by a general spectrally negative \lev process.  This allows the model to incorporate sudden declines of the project values, generalizing greatly the classical geometric Brownian motion model.  We solve the one-stage case as well as the extension to the multiple-stage case.  The optimal stopping times are of threshold-type and  the value function admits an expression in terms of the scale function. A series of numerical experiments are conducted to verify the optimality and to evaluate the efficiency of the algorithm.
\end{abstract}

\maketitle \noindent \small{\textbf{Key words:} 
Optimal multiple-stopping; Spectrally negative \lev processes; Real options  }\\

\section{Introduction}

Consider a firm facing a decision of when to abandon or contract a project so as to maximize the total expected future cash flows.  This problem is often referred to as the \emph{abandonment option} or the \emph{contraction option}.  A typical formulation reduces to a standard optimal stopping problem, where the uncertainty of the future cash flow is driven by a stochastic process and the objective is to find a stopping time that maximizes the total expected cash flows realized until then.  A more realistic extension is its multiple-stage version where the firm can withdraw from a project in stages.


In a standard formulation, given a discount rate $r > 0$ and $X_t = x+(\mu - \frac 1 2 \sigma^2) t + \sigma B_t$ for a standard Brownian motion $B$, $\mu \in \R$ and $\sigma > 0$ , one wants to obtain a stopping time $\tau$ of $X$ that maximizes the expectation
\begin{align}
\E \left[ \int_0^{\tau} e^{-rt}  (e^{X_t}-\delta)  \diff
t +  e^{-r \tau}
K  1_{\{\tau < \infty \}} \right]. \label{basic_model}
\end{align}
The profit collected continuously is modeled as the geometric Brownian motion $e^{X_t}$ less the constant operating expense $\delta \geq 0$. The value $K \in \R$ corresponds to the lump-sum benefits attained  (or the costs incurred) at the time of abandonment. Here a technical assumption $r > \mu$ is commonly imposed so that the expectation is finite and the problem is non-trivial.    The problem is rather simple mathematically; it reduces to the well-known perpetual American option (or the McKean optimal stopping problem).  An explicit solution can be attained even when $X$ is generalized to a \lev process (see, e.g., Mordecki \cite{mordecki}).

In this paper, we generalize the classical model by extending from Brownian motion to a general \lev process with negative jumps (spectrally negative \lev process),  and consider the optimal stopping problem of the form:
\begin{align}
\sup_{\tau} \E\left[   \int_0^{\tau} e^{-rt} f(X_t) \diff
t + e^{-r \tau}
g (X_{\tau}) 1_{\{\tau < \infty \}}\right]. \label{our_problem}
\end{align}
We obtain the optimal stopping time as well as the value function for the case $f$ is increasing and $g$ admits the form $g(x) = K- bx - \sum_{i=1}^N c_i e^{a_i x}$ for some positive constants $a$, $b$ and $c$.   We also show the optimality among all \emph{stopping times of threshold type} (see \eqref{def_tau_A} below) when $g$ is relaxed to be a general decreasing and concave function.  The decreasing property of $g$ reflects the fact that the cost of abandoning a project is higher when the project is large.

We further extend it to the multiple-stage case where one wants to obtain a set of stopping times $\{ \tau^{(m)}; 1 \leq m \leq M \}$ such that $0 =: \tau^{(0)} \leq \tau^{(1)} \leq \cdots \leq \tau^{(M)}$ a.s. and achieve
\begin{align}
\sup_{\tau^{(1)} \leq \cdots \leq \tau^{(M)}} \sum_{m=1}^M \E \left[   \int_{\tau^{(m-1)}}^{\tau^{(m)}} e^{-rt}  F_m(X_t) \diff t + e^{-r \tau^{(m)}}
g_m(X_{\tau^{(m)}})  1_{\{\tau^{(m)} < \infty \}} \right] \label{our_problem_mult}
\end{align}
when $g_m$ and $f_m := F_m - F_{m+1}$ (with $F_{M+1} := 0$), for each $1 \leq m \leq M$, satisfy the same assumptions as in the one-stage case. The multiple-stopping problem arises frequently  in real options (see e.g.\ \cite{dixit_pindyck}) and is well-studied particularly for the case $X$ is driven by Brownian motion.  In mathematical finance, similar problems are dealt in the valuation of swing options \cite{Carmona_dayanik, Touzi_Carmona} with refraction times between any consecutive stoppings.


Although the use of Brownian motion is fairly common in real options, empirical evidence suggests that the real world is not Gaussian, but with significant skewness and kurtosis (see, e.g., \cite{Boyarchenko_2007, Deaton_1992,Yang_Brorsen}).   Dixit and Pindyck \cite{dixit_pindyck} considered the case with  jumps of a fixed size with Poisson arrivals. Boyarchenko and Levendorski{\u\i} \cite{Boyarchenko_2006} considered the EPV approach for a general \lev process satisfying the (ACP)-condition (with a focus on exponential-type jumps for illustration); they solved a 
related multiple-stage problem with $g$ being constant.  
The \lev model  is in general less tractable than the continuous diffusion counterpart, especially when the lump-sum reward function $g$ is not a constant. 
When jumps are involved, the process can potentially jump over a threshold level, requiring one to compute the overshoot distributions that depend significantly on the form of the \lev measure.  Technical details are further required  when it has jumps of infinite activity or infinite variation.  For related literature, we refer the reader to, among others, \cite{alili-kyp, avram-et-al-2004,   Egami-Yamazaki-2010-1, Kyprianou_Surya_2007, Leung_Yamazaki_2010} for optimal stopping problems and \cite{Baurdoux2008,Baurdoux2009,  Leung_Yamazaki_2011, Hernandez_Yamazaki_2013} for optimal stopping games of spectrally negative \lev processes.  For a general reference on optimal stopping problems, see, e.g., \cite{peskir-shiryaev}.

In this paper, we take advantage of the recent advances in the theory of the spectrally negative \lev process (see, e.g., \cite{Bertoin_1997,Kyprianou_2006}).  In particular, we use the results by Egami and Yamazaki \cite{Egami-Yamazaki-2011}, where we obtained and showed the equivalence of the continuous/smooth fit condition and the first-order condition in a general optimal stopping problem.  
Unlike the two-sided jump case, the identification of the candidate optimal stopping time can be conducted efficiently without intricate computation. The resulting value function can be written in terms of the scale function, which also can be computed efficiently by using, e.g., \cite{Egami_Yamazaki_2010_2,Surya_2008}.   The extension to the multiple-stage case can be carried out without losing generality.  The resulting optimal stopping times are of threshold type with possibly simultaneous stoppings, and the value function again admits the form in terms of the scale function. 
 We also conduct a series of numerical experiments using the spectrally negative \lev process with phase-type  jumps so as to verify the optimality of the proposed strategies as well as the efficiency of the proposed algorithm.  



The rest of the paper is organized as follows.  In Section \ref{section_single_stopping}, we review the spectrally negative \lev process and the scale function and then solve the one-stage problem.  In Section \ref{section_multi_stage}, we extend it to the multiple-stage problem.  In Section \ref{section_numer}, we verify the optimality and efficiency of the algorithm through a series of numerical experiments.  Section \ref{section_conclusion} concludes the paper.

%

\section{One-stage Problem} \label{section_single_stopping}

Let $(\Omega, \F, \p)$ be a probability space hosting a
\emph{spectrally negative \lev process} $X=\{X_t: t\ge 0\}$
characterized uniquely by the \emph{Laplace exponent}
\begin{align}
\psi(\beta) := \log \E^0 \left[ e^{\beta X_1}\right] = c \beta
+\frac{1}{2}\sigma^2 \beta^2 +\int_{( 0,\infty)}(e^{-\beta
z}-1+\beta z 1_{\{0 <z<1\}})\,\Pi(\diff z), \quad  {\beta \geq 0}, \label{laplace_exp}
\end{align}
where $c \in\R$, $\sigma\geq 0$ and $\Pi$ is a \lev measure concentrated on
$(0,\infty)$ such that 
\begin{align}
\int_{(0,\infty)} (1  \wedge z^2) \Pi(
\diff z)<\infty. \label{integrability_levy_measure}
\end{align}
Here and throughout the paper $\p^x$ is the conditional probability where $X_0 = x \in \R$ and $\E^x$ is its expectation.   We exclude the case when
$X$ is a negative of a subordinator (i.e.\ it has monotone paths a.s.) and we shall further assume that the \lev measure is \emph{atomless}:
\begin{assump} \label{assump_no_atom}
We assume that $\Pi$ does not have atoms.
\end{assump}
In addition, we assume the following regarding the tail of the \lev measure.
\begin{assump}\label{assump_tail}
We assume that there exists some $\epsilon > 0$ such that
\begin{align*} 
\int_{[1,\infty)} e^{\epsilon u} \Pi (\diff u) < \infty.
\end{align*}
In particular, this guarantees that $\E^0 X_1 = \psi'(0+) \in (-\infty, \infty)$.
\end{assump}

This section considers the \emph{one-stage optimal stopping} problem of the form \eqref{our_problem}
where the supremum is taken over the set (or a subset) of stopping times with respect to the filtration $\mathbb{F} = (\mathcal{F}_t)_{t \geq 0}$ generated by $X$.  We assume the running payoff function $f$ to be \emph{increasing}.  The stochastic process $X$ models the state of the project and the monotonicity of $f$ means that it yields higher rewards when $X$ is high.  Typically one assumes $f(x) = e^x - \delta$ as in \eqref{basic_model} and this is clearly a special case of our model.  Regarding the terminal reward function $g$, we consider two cases: (i) when $g$ is a sum of linear and exponential functions (Assumption \ref{assump_optimality} below) and  (ii) when $g$ is a general decreasing and concave function (Assumption \ref{assump_optimality2} below).

The results discussed in this section are applications of  Egami and Yamazaki \cite{Egami-Yamazaki-2011} and will be extended to the multiple-stage problem in the next section.  Fix $r > 0$.  Let $\mathcal{S}$ be the set of all $[0,\infty]$-valued $\mathbb{F}$-stopping times and define for any $\tau \in \S$,
\begin{align}
u(x, \tau) \equiv u(x, \tau; f,g)  := \E^x \left[ \int_0^{\tau} e^{-rt} f(X_t) \diff
t + e^{-r \tau}
g (X_{\tau}) 1_{\{\tau < \infty \}}\right], \quad x \in \R. \label{u_notation}
\end{align}
After a brief review on the scale function and the results of \cite{Egami-Yamazaki-2011}, we shall solve, under Assumption \ref{assump_optimality} below, the problem:
\begin{align*}
u(x) := \sup_{\tau \in \S} u(x, \tau).
\end{align*}
We then obtain under Assumption \ref{assump_optimality2} below a weaker version of optimality:
\begin{align*}
\widetilde{u}(x) := \sup_{\tau \in \widetilde{\S}} u(x, \tau),
\end{align*}
over the set of all \emph{first down-crossing times},
\begin{align*}
\widetilde{\mathcal{S}} := \{ \tau_A: A \in \R \},
\end{align*}
where
\begin{align}
\tau_A := \inf \left\{ t > 0: X_t \leq A \right\}, \quad A \in \R, \label{def_tau_A}
\end{align}
with $\inf \emptyset = \infty$ by convention.
This form of optimality is often used in real options and also in the field of corporate finance and credit risk as exemplified by Leland's endogenous default model \cite{Leland_1994, Leland_Toft_1996}.  In practice, a strategy must be simple enough to implement and it is in many cases a reasonable assumption to focus on the set of stopping times of threshold type as in \eqref{def_tau_A}. 
Because $\widetilde{\S} \subset \S$, it is clear that $u \geq \widetilde{u}$.
For the rest of the paper, let $h_\pm(x) := \pm h(x) \vee 0$, $x \in \R$, for any measurable function $h: \R \rightarrow \R$.

\subsection{Review of scale functions and Egami and Yamazaki \cite{Egami-Yamazaki-2011}} \label{subsection_review}
For any spectrally negative \lev process, there exists a function called  the  \emph{(r-)scale function} 
\begin{align*}
W^{(r)}: \R \rightarrow [0,\infty) , 
\end{align*}
which is zero on $(-\infty,0)$, continuous and strictly increasing on $[0,\infty)$, and is characterized by the Laplace transform:
\begin{align*}
\int_0^\infty e^{-s x} W^{(r)}(x) \diff x = \frac 1
{\psi(s)-r}, \qquad s > \lapinv,
\end{align*}
where
\begin{equation}
\lapinv :=\sup\{\lambda \geq 0: \psi(\lambda)=r\}, \quad
r\ge 0. \notag
\end{equation}
Here, the Laplace exponent $\psi$ in \eqref{laplace_exp} is known to be zero at the origin and convex on $[0,\infty)$; $\lapinv$ is well-defined and is strictly positive whenever $r > 0$.  We also define the second scale function:
\begin{align*}
Z^{(r)}(x) := 1 + r \int_0^x W^{(r)}(y) \diff y, \quad x \in \R.
\end{align*}
As we shall see below, the pair of scale functions $W^{(r)}$ and $Z^{(r)}$ play significant roles in our problems; for a comprehensive account of the scale function, we refer the reader to, e.g., \cite{Bertoin_1997,Kuznetsov2013, Kyprianou_2006}.

Recall \eqref{def_tau_A} and define the first up-crossing times of $X$ by $\tau_b^+ := \inf \left\{ t \geq 0: X_t \geq b \right\}$.
Then, for any $b > 0$ and $0 < x \leq b$, as summarized in Theorem 8.1 of \cite{Kyprianou_2006},
\begin{align}
\begin{split}
\E^x \left[ e^{-r \tau_b^+} 1_{\left\{ \tau_b^+ < \tau_0 \right\}}\right] &= \frac {W^{(r)}(x)}  {W^{(r)}(b)}, \\
 \E^x \left[ e^{-r \tau_0} 1_{\left\{ \tau_b^+ > \tau_0 \right\}}\right] &= Z^{(r)}(x) -  Z^{(r)}(b) \frac {W^{(r)}(x)}  {W^{(r)}(b)}, \\
 \E^x \left[ e^{-r \tau_0} \right] &= Z^{(r)}(x) -  \frac r {\Phi_r} W^{(r)}(x).
\end{split}
 \label{laplace_in_terms_of_z}
\end{align}

As in Lemmas 8.3 and 8.5 of  Kyprianou \cite{Kyprianou_2006},  for each $x > 0$, the functions $r \mapsto W^{(r)}(x)$ and $r \mapsto Z^{(r)}(x)$ can be analytically extended to $r \in \mathbb{C}$.  
Fix $a \geq 0$ and define $\psi_a(\cdot)$, as the Laplace exponent of $X$ under $\p_a$ with the change of measure $ \left. \frac {\diff \p_a^0} {\diff \p^0}\right|_{\mathcal{F}_t} = \exp(a X_t - \psi(a) t)$, $t \geq 0$; as in page 213 of \cite{Kyprianou_2006}, for all $\beta > -a$,
\begin{align}
\psi_a(\beta) :=\Big(  a \sigma^2  + c - \int_{(0,1)} u (e^{-au}-1) \Pi(\diff u)\Big) \beta
+\frac{1}{2}\sigma^2 \beta^2 +\int_{(0,\infty)} (e^{- \beta u}-1 + \beta u 1_{\{ 0 < u < 1 \}}) e^{-a u}\,\Pi(\diff u).\label{psi_a}
\end{align}

 If $W_a$ and $Z_a$ are the scale functions associated with $X$ under $\p_a$ (or equivalently with $\psi_a(\cdot)$).  
Then, by Lemma 8.4 of \cite{Kyprianou_2006},
\begin{align}
W_a^{(r-\psi(a))}(x) = e^{-a x} W^{(r)}(x), \quad x \geq 0. \label{scale_measure_change}
\end{align}
In particular, by setting $a = \Phi_r$ (or equivalently $r = \psi(a)$), we can define
\begin{align}
W_{\Phi_r} (x) := W_{\Phi_r}^{(0)} (x)  = e^{-\Phi_r x} W^{(r)} (x), \quad x \in \R \label{W_scaled}
\end{align}
which satisfies
\begin{align*}
\int_0^\infty e^{-\beta x} W_{\Phi_r} (x) \diff x &= \frac 1
{\psi(\beta+\Phi_r)-r}, \quad \beta > 0.
\end{align*}

The smoothness and asymptotic behaviors around zero of the scale function are particularly important in our analysis.  We summarize these in the remark given immediately below.
\begin{remark} \label{remark_smoothness_zero}
\begin{enumerate}
\item Assumption \ref{assump_no_atom} guarantees that $W^{(r)}$ is $C^1$ on $(0,\infty)$.  In particular, when $\sigma > 0$, then $W^{(r)}$ is $C^2$ on $(0,\infty)$.  Fore more details on the smoothness of the scale function, see Chan et al.\ \cite{Chan_2009}.
\item As in Lemmas 4.3 and 4.4 of 
\cite{Kyprianou_Surya_2007},
\begin{align*}
W^{(r)} (0) &= \left\{ \begin{array}{ll} 0, & \textrm{if $X$  is  of unbounded
variation} \\ \frac 1 {\mu}, & \textrm{if $X$  is  of bounded variation}
\end{array} \right\}, \\ W^{(r)'} (0+) &:= \lim_{x \downarrow 0}W^{(r)'} (x) =
\left\{ \begin{array}{ll}  \frac 2 {\sigma^2}, & \textrm{if  }  \sigma > 0 \\
\infty, & \textrm{if  }  \sigma = 0 \; \textrm{and} \; \Pi(0,\infty) = \infty \\
\frac {r + \Pi(0,\infty)} {\mu^2}, & \textrm{if $X$  is compound Poisson}
\end{array} \right\},
\end{align*}
where $\mu := c + \int_{( 0,1)}z\, \Pi(\diff z)$,
which is finite when $X$ is of bounded variation.
\end{enumerate}
\end{remark}

In \cite{Egami-Yamazaki-2011}, we have shown that a candidate optimal stopping time can be efficiently identified using the scale function.  Define the expected payoff corresponding to the down-crossing time \eqref{def_tau_A} by
\begin{align}
u_A(x) &:= u(x, \tau_A), \quad x, A \in \R, \label{u_A}
\end{align}
which equals $g(x)$ for $x \leq A$.
By combining the compensation formula for \lev processes and the resolvent measure written in terms of the scale function, this can be written in a semi-explicit form.  Let \begin{align}
\Psi_f(A) &:= \int_0^\infty e^{-\lapinv y} f(y+A)\diff y, \quad A \in \R, \label{big_H} \\
\Theta_f(x; A) &:= \int_{A}^x  W^{(r)}(x-y) f(y) \diff y, \quad x, A \in \R, \label{big_Theta}
\end{align}
and
\begin{align} \label{def_rhos}
\begin{split}
\rho_{g,A}^{(r)} &:= \int_{(0,\infty)} \Pi (\diff u)   \int_0^{u} e^{-\Phi_r z} (g(z+A-u)-g(A)) \diff z  \\ &\equiv \int_{(0,\infty)} \Pi (\diff u)   \int_A^{u+A} e^{-\Phi_r (y-A)} (g(y-u)-g(A)) \diff y, \quad A \in \R, \\
\varphi_{g,A}^{(r)}(x) &:=\int_{(0,\infty)} \Pi (\diff u)  \int_0^{u \wedge (x-A)} W^{(r)} (x-z-A) (g(z+A-u)-g(A)) \diff z, \quad x > A.
\end{split}
\end{align}

These integrals are well-defined if  
\begin{align}
&\int_0^\infty e^{-\lapinv y} |f(y+A) | \diff y < \infty, \quad A \in \R, \label{cond_f_integrability} \\
&g \in C^2, \quad \textrm{and} \quad \int_{[1,\infty)} \Pi (\diff u) \max_{A-u \leq \zeta \leq A} |g(\zeta)-g(A)| < \infty, \quad A \in \R. \label{cond_g_integrability}
\end{align}
If these are satisfied, we can write $u_A(x)$ as in  \eqref{u_A} for $x > A$ as the sum of the following three terms:
\begin{align}
\begin{split}
\Gamma_1(x;A) &:=  g(A) \left[ Z^{(r)}(x-A) - \frac r {\Phi_r} W^{(r)}(x-A) \right], \\
\Gamma_2(x;A) &:=  W^{(r)} (x-A) \rho_{g,A}^{(r)} - \varphi_{g,A}^{(r)}(x),  \\
\Gamma_3(x;A) &:=
W^{(r)} (x-A) \Psi_f(A)  -
\Theta_f(x; A).
\end{split} \label{Gammas}
\end{align}

Egami and Yamazaki \cite{Egami-Yamazaki-2011} obtained the \emph{first-order condition} that makes $ \partial u_A(x) / \partial A$ vanish and showed that it is equivalent to the \emph{continuous fit} condition $u_A(A+) := \lim_{x \downarrow A} u_A(x) = g(A)$ when $X$ is of bounded variation and to the \emph{smooth fit} condition $u_A'(A+) := \lim_{x \downarrow A} u_{A}'(x)  = g'(A)$ when $\sigma > 0$.  Recall that $X$ is of bounded variation if and only if $\sigma = 0$ and \begin{align}
\int_{(0,\infty)} (1 \wedge z)\, \Pi(\diff z) < \infty.
\label{cond_bounded_variation}
\end{align} 
It has been shown that
\begin{align}
u_A(A+) = g(A) +  W^{(r)}(0) \Lambda(A), \quad A \in \R,  \label{cond_continuous_fit}
\end{align}
where
\begin{align}
\Lambda(A) \equiv \Lambda(A; f,g ) :=  -\frac r {\lapinv} g(A)  - \frac {\sigma^2} 2 g'(A) +  \rho_{g,A}^{(r)}   + \Psi_f(A), \quad A \in \R.\label{Large_lambda}
\end{align}
In view of Remark \ref{remark_smoothness_zero}(2), for the unbounded variation case, continuous fit holds whatever the choice of $A$ is, while, for the bounded variation case, it holds if and only if $\Lambda(A)=0$.

Furthermore,  it has been shown by \cite{Egami-Yamazaki-2011}, on condition that there exists some $\delta > 0$ satisfying
\begin{align}
\int_{[1,\infty)} \Pi (\diff u)\sup_{0 \leq \xi \leq \delta} |g(A+\xi) - g(A+\xi-u)|  < \infty, \label{cond_g_integrability2}
\end{align}
we have
\begin{align}
\frac \partial {\partial A} u_A(x) = -  e^{\lapinv (x-A)} W_{\lapinv}' (x-A)  \Lambda(A), \quad x > A, \label{u_A_derivative}
\end{align}
where $W_{\lapinv}$ is defined in \eqref{W_scaled}.
 It is known that $W_{\lapinv}$ is increasing and hence, if $\Lambda(A)$ is monotonically increasing, the down-crossing time $\tau_A$ for such $A$ with $\Lambda(A)=0$ becomes a natural candidate for the optimal stopping time.

\subsection{Exponential/Linear Case}We first consider the case where $g$ admits the form:
\begin{align}
g(x) = K-   bx - \sum_{i=1}^N c_i e^{a_i x}, \quad x \in \R, \label{g_polynomial}
\end{align}
for some constants $K \in \R$, $b \geq 0$ and $c_i, a_i > 0$, $1 \leq i \leq N$, $N \geq 0$.   We assume without loss of generality that $a_i \neq a_j$ for $i \neq j$.  The conditions \eqref{cond_g_integrability} and \eqref{cond_g_integrability2} are satisfied by Assumption \ref{assump_tail}.  For $f$, we need a technical condition so that \eqref{cond_f_integrability} is  guaranteed.  We summarize the conditions in the Assumption given below.

\begin{assump} \label{assump_optimality}
We assume the following.
\begin{enumerate}
\item $f(\cdot)$ is continuous, piecewise differentiable, and increasing.  In addition, the growth of $f_-$ as $x \downarrow -\infty$ is at most polynomial and $\int_0^\infty e^{-\lapinv y} f_+(y+x )  \diff y < \infty$, $x \in \R$; these guarantee \eqref{cond_f_integrability}.
\item $g(\cdot)$ admits the form \eqref{g_polynomial} for some $K \in \R$, $b \geq 0$, and strictly positive constants $a_i$ and $c_i$, $1 \leq i \leq N$, $N \geq 0$ such that $a_i \neq a_j$ for any $i \neq j$, 
\end{enumerate}
\end{assump}

\begin{remark}  \label{remark_finiteness_of_f}Assumptions \ref{assump_tail} and \ref{assump_optimality}(1) guarantee that  $\E^x \left[ \int_0^\infty e^{-rt} f_-(X_t) \diff
t  \right] < \infty$ for all $x \in \R$; for its proof, see, e.g., \cite{Yamazaki_2013}. By this and  Corollary 8.9 of \cite{Kyprianou_2006}, 
\begin{align}
\E\left[   \int_0^{\infty} e^{-rt} f(X_t) \diff
t  \right] =  \int_{-\infty}^\infty \Big( \Phi_r'  e^{-\lapinv(y-x)} - W^{(r)}(x-y) \Big) f(y) \diff y \label{value_function_infty}
\end{align}
exists, where $\Phi_r'$ is the derivative of $\Phi_r$ with respect to $r$.

Moreover, this is finite. Indeed, $\E\left[   \int_0^{\infty} e^{-rt} f_+(X_t) 1_{\{X_t \geq 0 \}}\diff
t  \right] =  \int_0^\infty \Big( \Phi_r'  e^{-\lapinv(y-x)} - W^{(r)}(x-y) \Big) f_+(y) \diff y < \infty$
 by Assumption \ref{assump_optimality}(1) and because $W^{(r)}$ is zero on the negative half line. On the other hand, because $f$ is increasing, $\E\left[   \int_0^{\infty} e^{-rt} f_+(X_t) 1_{\{X_t < 0 \}}\diff
t  \right] \leq f_+(0) \E\left[   \int_0^{\infty} e^{-rt} 1_{\{X_t < 0 \}}\diff t  \right]  \leq   {f_+(0)} /{r}$.
\end{remark}




With Assumption \ref{assump_optimality}, we simplify \eqref{Large_lambda} using
\begin{align}
\varpi_r(a) := \left\{ \begin{array}{ll} \frac {r- \psi(a)} {\lapinv - a}, & a \neq \Phi_r\\ \psi'(\Phi_r) = \lim_{a \rightarrow \lapinv}\frac {r- \psi(a)} {\lapinv - a}, &  a = \Phi_r \end{array} \right\}, \quad a > 0. \label{varpi}
\end{align}
By the convexity of $\psi$, $\varpi_r(a) > 0$ for any $a > 0$.  The proof of the following lemma is given in Appendix \ref{proof_lemma_lambda_A_polynomial}.
\begin{lemma} \label{lemma_lambda_A_polynomial}
For every $A \in \R$, we have
\begin{align} 
\Lambda(A) = -\frac r {\lapinv}K +    b \Big(\frac r {\Phi_r^2}+ \frac {rA -\psi'(0+) } {\Phi_r}\Big) + \sum_{i=1}^N c_i e^{a_i A} \varpi_r (a_i)  + \Psi_f(A). \label{lambda_A_polynomial}
\end{align}
\end{lemma}

In view of \eqref{lambda_A_polynomial} above, the function $\Lambda(A)$
 is clearly continuous and  increasing.   
Therefore, if $\lim_{A \downarrow -\infty}\Lambda(A) < 0 < \lim_{A \uparrow \infty}\Lambda(A)$, there exists a unique root $A^* \in \R$ such that $\Lambda(A^*) = 0$.  Otherwise, let $A^*=-\infty$ if $\lim_{A \downarrow -\infty}\Lambda(A) \geq 0$ and let $A^* = \infty$ if $\lim_{A \uparrow \infty}\Lambda(A) \leq 0$.  

\begin{remark} \label{remark_A_tilde} 
Except for the case $g$ is a constant, because $\Lambda(A)$ increases to $\infty$, we have $A^* < \infty$.
\end{remark}

With our assumption on the form of $g$, the value function can be written succinctly.   By Proposition 2 of Avram et al. \cite{Avram_et_al_2007} and because $\psi'(0+) \in (-\infty, \infty)$ by Assumption \ref{assump_tail},
\begin{align*}
\E^x [e^{-r \tau_0} X_{\tau_0}] = \overline{Z}^{(r)}(x) - \psi'(0+) \frac {Z^{(r)}(x)-1} r- \frac {r-\psi'(0+)\Phi_r} {\Phi_r^2} W^{(r)}(x), \quad x \in \R,
\end{align*}
where
\begin{align*}
\overline{Z}^{(r)}(x) := \int_0^x  Z^{(r)}(y) \diff y, \quad x \in \R.
\end{align*}
This together with \eqref{laplace_in_terms_of_z} gives, for any $x, A \in \R$,
\begin{align*}
\E^x [e^{-r \tau_A} X_{\tau_A}] = \overline{Z}^{(r)}(x-A) + \Big(A- \frac {\psi'(0+)} r \Big)Z^{(r)}(x-A) + \frac {\psi'(0+)} r - \frac {r-\psi'(0+)\Phi_r + r A \Phi_r} {\Phi_r^2} W^{(r)}(x-A).
\end{align*}

With the help of Exercise 8.7(ii) of \cite{Kyprianou_2006}, the expression \eqref{u_A}  becomes
\begin{align} \label{value_function_for_any_A}
\begin{split}
u_{A}(x) &= K \Big( Z^{(r)} (x-A) - \frac r {\lapinv} W^{(r)}(x-A) \Big)- \sum_{i=1}^N c_i e^{a_i x} \Big( Z_{a_i}^{(r - \psi(a_i))} (x-A) - \varpi_r(a_i) W_{a_i}^{(r-\psi(a_i))} (x-A) \Big) \\
&- b \Big[ \overline{Z}^{(r)}(x-A) + \Big(A- \frac {\psi'(0+)} r \Big)Z^{(r)}(x-A) + \frac {\psi'(0+)} r - \frac {r-\psi'(0+)\Phi_r + r A \Phi_r} {\Phi_r^2} W^{(r)}(x-A) \Big]
\\ &+ W^{(r)} (x-A) \Psi_f(A) - \Theta_f(x;A).
\end{split}
\end{align}
Moreover, if $A^* \in (-\infty, \infty)$, by how $A^*$ is chosen as in \eqref{lambda_A_polynomial} and by \eqref{scale_measure_change}, it can be simplified to
\begin{align}
\begin{split}
u_{A^*}(x) = K Z^{(r)} (x-A^*) - b \Big[ \overline{Z}^{(r)}(x-A^*) + \big(A^*- \frac {\psi'(0+)} r \big)Z^{(r)}(x-A^*) + \frac {\psi'(0+)} r \Big]\\
- \sum_{i=1}^N c_i e^{a_i x} Z_{a_i}^{(r - \psi(a_i))} (x-A^*) - \Theta_f(x; A^*).
\end{split} \label{value_function}
\end{align}
The verification of optimality requires the following smoothness properties, whose  proofs are given in Appendix \ref{proof_lemma_smoothness}.
\begin{lemma} \label{lemma_smoothness} Suppose $-\infty < A^* \leq \infty$.
\begin{enumerate}
\item $u_{A^*}(x)$ is $C^1$ on $\R \backslash \{A^*\}$.
\item In particular, when $X$ is of unbounded variation, $u_{A^*}(x)$ is $C^2$ on $\R \backslash \{A^*\}$.
\end{enumerate}
\end{lemma}

Herein, we add a remark concerning continuous/smooth fit.  
The following remark confirms the results in \cite{Egami-Yamazaki-2011} and further verifies that smooth fit holds whenever $X$ is of unbounded variation even when $\sigma = 0$.  This observation only requires the asymptotic behavior of the scale function near zero as in Remark \ref{remark_smoothness_zero}(2).
\begin{remark}[continuous/smooth fit] \label{remark_fit} Suppose $-\infty < A^* < \infty$.
\begin{enumerate}
\item Continuous fit holds (i.e.\ $u_{A^*}(A^*+) = g(A^*)$) because, by \eqref{value_function}, $Z^{(r)}(0) = Z_{a_i}^{(r - \psi(a_i))} (0) = 1$ and $\lim_{x \downarrow A^*}\Theta_f(x; A^*) = 0$.
\item In particular, when $X$ is of unbounded variation, smooth fit holds  (i.e.\ $u'_{A^*}(A^*+) = g'(A^*)$) because 
\begin{align*}
u_{A^*}'(x) &= K r W^{(r)} (x-A^*) - b \Big[ Z^{(r)}(x-A^*) + r \big(A^*- \frac {\psi'(0+)} r \big)W^{(r)}(x-A^*)  \Big]\\ &- \sum_{i=1}^N c_i e^{a_i x} (r - \psi(a_i))  W_{a_i}^{(r - \psi(a_i))} (x-A^*) - \sum_{i=1}^N a_i c_i e^{a_i x}  Z_{a_i}^{(r - \psi(a_i))} (x-A^*) - \Theta_f'(x; A^*) 
\\
&\xrightarrow{x \downarrow A^*}  - b- \sum_{i=1}^N a_i c_i e^{a_i A^*} = g'(A^*),
\end{align*}
thanks to $W^{(r)}(0) = W_{a_i}^{(r - \psi(a_i))} (0)= 0$, $Z^{(r)}(0) =  Z_{a_i}^{(r - \psi(a_i))} (0) = 1$ and $\lim_{x \downarrow A} \Theta_f'(x;A) = 0$; see also the proof of Lemma \ref{lemma_smoothness}. 
\end{enumerate}
\end{remark}

We now state the main results of this subsection.   The proof is given in Appendix \ref{proof_theorem_polynomial}.
\begin{proposition} \label{theorem_polynomial}
\begin{enumerate}
\item If $-\infty < A^* < \infty$, the stopping time $\tau_{A^*} := \inf \left\{ t \geq 0: X_t \leq A^* \right\}$
is optimal over $\S$ and the value function is $u(x)  = u_{A^*}(x)$ as in \eqref{value_function} for all $x \in \R$.

\item If $A^*=\infty$, immediate stopping is always optimal and $u(x) = u_{\infty}(x) := g(x)$ for any $x \in \R$.
\item If $A^* = -\infty$, it is never optimal to stop (i.e.\ $\tau^* = \infty$ a.s.\ is optimal), and the value function is $u(x) = u_{-\infty}(x)$ that is given in \eqref{value_function_infty}.
\end{enumerate}
\end{proposition}

\subsection{For a general concave and decreasing $g$} \label{subsection_single_general}

We now relax the assumption on $g$ and consider a general concave and decreasing function $g$.  
We also drop the continuity assumption on $f$.
\begin{assump} \label{assump_optimality2}
We assume the following.
\begin{enumerate}
 \item $f(\cdot)$ is  increasing.  In addition, the growth of $f_-$ as $x \downarrow -\infty$ is at most polynomial and $\int_0^\infty e^{-\lapinv y} f_+(y+x )  \diff y < \infty$, $x \in \R$.
\item $g(\cdot)$ is twice-differentiable, concave and monotonically  decreasing such that \eqref{cond_g_integrability} and \eqref{cond_g_integrability2} hold.
\end{enumerate}
\end{assump}
Under this assumption, we see that $\Lambda(A)$ as in \eqref{Large_lambda} is continuous and increasing.  Indeed, we have 
\begin{align*}
\frac \partial {\partial A} \left[ - \frac r {\Phi_r} g(A) - \frac {\sigma^2} 2 g'(A)\right] =  -\frac r {\lapinv} g'(A)  - \frac {\sigma^2} 2 g''(A) \geq 0.
\end{align*}
Moreover,  $\rho^{(r)}_{g,A}$ is increasing by the concavity of $g$.  On the other hand, $\Psi_f(A)$ is increasing because $f$ is.
%
 Therefore, we again define $A^*$ in the same way as  the unique root of $ \Lambda(A)=0$ (if it exists).  
The proof of the following result is given in Appendix \ref{proof_proposition_general}.

%
%
%

\begin{proposition} \label{proposition_general}
Suppose Assumption \ref{assump_optimality2}.
\begin{enumerate}
\item When $-\infty < A^* < \infty$,  then $\tau_{A^*}$ is optimal over $\widetilde{\mathcal{S}}$ and the value function is given by  
\begin{align}
\widetilde{u}(x) = u_{A^*}(x) =  g(A^*)  Z^{(r)}(x-A^*) + W^{(r)} (x-A^*)  \frac {\sigma^2} 2 g'(A^*) - \varphi_{g,A^*}^{(r)}(x)  -
\Theta_f(x; A^*), \quad x > A^*. \label{u_tilde}
\end{align}
For $x \leq A^*$, we have $\widetilde{u}(x) = g(x)$.
\item If $A^*=\infty$, immediate stopping is always optimal and $\widetilde{u}(x) = u_{\infty}(x) :=g(x)$ for any $x \in \R$.
\item If $A^*=-\infty$, then $\tau^*=\infty$ a.s.\  is optimal over $\S$ and the value function is $u(x) = u_{-\infty}(x)$ that is given in \eqref{value_function_infty}.
\end{enumerate}
\end{proposition}

%


\section{Multiple-stage problem} \label{section_multi_stage}
In this section, we extend to the scenario the firm can decrease its involvement in the project in multiple stages as defined in \eqref{our_problem_mult}.  As in the one-stage case, we consider two modes of optimality:
\begin{align}
U^{(M)}(x) &:= \sup_{(\tau^{(1)}, \ldots, \tau^{(M)}) \in \S_M } \sum_{m=1}^M \E^x \left[  \int_{\tau^{(m-1)}}^{\tau^{(m)}} e^{-rt}  F_m(X_t) \diff t +  e^{-r \tau^{(m)}}
g_m(X_{\tau^{(m)}})  1_{\{\tau^{(m)} < \infty \}} \right], \label{multi_stage_objective} \\
\widetilde{U}^{(M)}(x) &:= \sup_{(\tau^{(1)}, \ldots, \tau^{(M)}) \in \ST_M } \sum_{m=1}^M \E^x \left[  \int_{\tau^{(m-1)}}^{\tau^{(m)}} e^{-rt}  F_m(X_t) \diff t +  e^{-r \tau^{(m)}}
g_m(X_{\tau^{(m)}})  1_{\{\tau^{(m)} < \infty \}} \right],  \label{multi_stage_objective_tilde}
\end{align}
for all $x \in \R$ where we define $\tau^{(0)} := 0$ for notational brevity and the supremum is, respectively, over the set of increasing sequences of $M$ stopping times,
\begin{align*}
\S_M := \{ \tau^{(m)} \in \S, 1 \leq m \leq M: \tau^{(1)} \leq \cdots \leq \tau^{(M)} \},
\end{align*}
and over the set of increasing sequences of $M$ down-crossing times,
\begin{align*}
\ST_M := \{ \tau^{(m)} = \tau_{A_m}\in \ST, 1 \leq m \leq M:  A_1 \geq A_2 \cdots \geq A_M \}.
\end{align*}
Clearly, $\ST_M \subset \S_M $ and hence $\widetilde{U}^{(M)} \leq U^{(M)}$.

We first consider the case $g_m$ admits the form
\begin{align}
g_m(x) := K_m - b_m x - \sum_{i=1}^{N_m} c_{mi} e^{a_{mi} x} \quad 1 \leq m \leq M, \label{cond_g}
\end{align}
for some constants $K_m \in \R$, $b_m \geq 0$, $c_{mi}, a_{mi} > 0$, $1 \leq i \leq N_m$,
and show the optimality in the sense of  \eqref{multi_stage_objective}  as an extension of Proposition \ref{theorem_polynomial}.  We then consider a more general case where $g_m$  is twice-differentiable, concave and monotonically decreasing and show the optimality over $\ST_M$ as an extension of Proposition \ref{proposition_general}. Regarding the running reward function $F$, define the differences:
\begin{align*}
f_m := F_{m} - F_{m+1}, \quad 1 \leq m \leq M,
\end{align*}
with $F_{M+1} \equiv 0$.   As is also assumed in \cite{Boyarchenko_2006}, we consider the case  $f_m$ is increasing for each $m$.
Using the notation as in \eqref{u_notation}, we can then write for all $x \in \R$
\begin{align}
U^{(M)}(x) = \sup_{(\tau^{(1)}, \ldots, \tau^{(M)}) \in \S_M }  \sum_{m=1}^M  u(x, \tau^{(m)}; f_m, g_m), \label{multi_stage_objective2} \\
\widetilde{U}^{(M)}(x) = \sup_{(\tau^{(1)}, \ldots, \tau^{(M)}) \in \ST_M }  \sum_{m=1}^M u(x, \tau^{(m)}; f_m, g_m). \label{multi_stage_objective2_tilde}
\end{align}
In summary, we assume Assumptions \ref{assump_h_multi} and \ref{assump_h_multi2} below for \eqref{multi_stage_objective} and \eqref{multi_stage_objective_tilde}, respectively.
\begin{assump} \label{assump_h_multi}
For each  $1 \leq m  \leq M$, we assume that $f_m$ and $g_m$ satisfy Assumption \ref{assump_optimality}.  
\end{assump}
\begin{assump} \label{assump_h_multi2}
For each  $1 \leq m  \leq M$, we assume that $f_m$ and $g_m$ satisfy Assumption \ref{assump_optimality2}.  
\end{assump}

As is clear from the problem structure, simultaneous stoppings (i.e.\ $\tau_k = \cdots = \tau_{k+l}$ a.s.\ for some $k$ and $l$) may be optimal in case it is not advantageous to stay in some intermediate stages.  For this reason,  define, for any subset $\II = \{ \min \II, \min \II + 1, \ldots, \max \II \}\subset \{1, \ldots, M\}$,
\begin{align}
g_{\II} := \sum_{i \in \II} g_i \quad \textrm{and} \quad  f_{\II} := F_{\min \II} - F_{\max \II+1},  \label{def_g_f}
\end{align} 
and consider an \emph{auxiliary one-stage problem} \eqref{our_problem} with $g = g_{\II}$ and $f = f_{\II}$.
Notice that Assumption \ref{assump_h_multi} (resp.\  Assumption \ref{assump_h_multi2}) guarantees that $f_{\II}$ and  $g_{\II}$ also satisfy Assumption \ref{assump_optimality} (resp.\ Assumption \ref{assump_optimality2}) for any $\II$. 
%
Hence Propositions \ref{theorem_polynomial} and  \ref{proposition_general} apply.   

Let 
\begin{align}
\Lambda_m(A) :=  \Lambda(A; f_m, g_m), \quad A \in \R, 1 \leq m \leq M,
\end{align}
as the function \eqref{Large_lambda} for $(f_m,g_m)$.  Because $\rho_{h_1 + h_2,A}^{(r)} \equiv \rho_{h_1,A}^{(r)} + \rho_{h_2,A}^{(r)}$ and $\Psi_{h_1 + h_2}(A) \equiv \Psi_{h_1}(A)+\Psi_{h_2}(A)$ for any measurable functions $h_1$ and $h_2$ in view of \eqref{big_H}  and \eqref{def_rhos}, we see that 
\begin{align}
\Lambda_{\II}(A) :=  \Lambda(A; f_{\II}, g_{\II}) = \Lambda\Big(A; \sum_{m \in \II}f_m, \sum_{m \in \II} g_m \Big) = \sum_{m \in \II} \Lambda_m (A) \label{Lambda_II}
\end{align}
is increasing and corresponds to the function \eqref{Large_lambda} for $(f_{\II},g_{\II})$.   In particular, under Assumption \ref{assump_optimality}, this reduces by Lemma \ref{lemma_lambda_A_polynomial} to
\begin{align*}
\Lambda_{\II}(A) &= \sum_{m \in \II} \Big[ -\frac r {\lapinv}K_m +  b_m \Big(\frac r {\Phi_r^2}+ \frac {rA -\psi'(0+) } {\Phi_r}\Big)  +  \sum_{i=1}^{N_m} c_{mi} e^{a_{mi} A}\varpi_r(a_{mi}) + \Psi_{f_m} (A) \Big]. 
\end{align*}
Now let $A_{\II}^*$ be the root of $\Lambda_{\II}(A)=0$ if it exists. If  $\lim_{A \uparrow \infty}\Lambda_{\II}(A) \leq 0$, we set $A_{\II}^*=\infty$; if $\lim_{A \downarrow -\infty}\Lambda_{\II}(A) \geq 0$, we set $A_{\II}^*=-\infty$.   For simplicity, let  $A_m^* := A_{\{m\}}^*$ for any $1 \leq m \leq M$.   Also 
define
\begin{align*}
u_{A}^{\II}(x)&:= u(x, \tau_A; f_{\II}, g_{\II}), \quad x, A \in \R.
\end{align*}
With these notations, the following is immediate by Propositions \ref{theorem_polynomial} and  \ref{proposition_general}.
\begin{corollary} \label{corollary_w}
Fix any $\II$ and $x \in \R$, and consider the problems:
\begin{align*}
u_{\II}(x) := \sup_{\tau \in \S}u(x, \tau; f_{\II}, g_{\II}) \quad \textrm{and} \quad \widetilde{u}_{\II}(x) := \sup_{\tau \in \ST}u(x, \tau; f_{\II}, g_{\II}). 
\end{align*}
Suppose Assumption \ref{assump_h_multi}. 
\begin{enumerate}
\item If $-\infty < A_{\II}^* < \infty$,  then 
\begin{align*}
u_{\II}(x)  = u_{A^*_{\II}}^{\II}(x) = \sum_{m \in \II} \Big( K_m Z^{(r)} (x-A_{\II}^*) - b_m \Big[ \overline{Z}^{(r)}(x-A_{\II}^*) + \big(A_{\II}^*- \frac {\psi'(0+)} r \big)Z^{(r)}(x-A_{\II}^*) + \frac {\psi'(0+)} r \Big] \\ - \sum_{i=1}^{N_m} c_{mi} e^{a_{mi} x} Z_{a_{mi}}^{(r - \psi(a_{mi}))} (x-A_{\II}^*) \Big) - \Theta_{f_{\II}}(x; A_{\II}^*),
\end{align*}
and
the stopping time $\tau_{A^*_{\II}} := \inf \left\{ t > 0: X_t \leq A^*_\II \right\}$
is optimal.

\item If $A^*_{\II}=\infty$, $u_{\II}(x) = g_{\II}(x)$ for any $x \in \R$ with the optimal stopping time $\tau^* = 0$ a.s.
\item If $A_{\II}^* = -\infty$, it is never optimal to stop, and the value function is given by
\begin{align}
u_{\II}(x) = \int_{-\infty}^\infty \big( \Phi'_r e^{-\lapinv(y-x)} - W^{(r)}(x-y) \big) f_{\II}(y) \diff y. \label{u_I}
\end{align}
\end{enumerate}
Suppose Assumption \ref{assump_h_multi2}. 
\begin{enumerate}
\item If $-\infty < A^*_{\II} < \infty$,  then  $\tau_{A^*_{\II}}$ is optimal and 
\begin{align*}
\widetilde{u}_{\II}(x) = u_{A^*_{\II}}^{\II}(x) =  g_{\II}(A^*_{\II})  Z^{(r)}(x-A^*_{\II}) + W^{(r)} (x-A^*_{\II})  \frac {\sigma^2} 2 g_{\II}'(A^*_{\II}) - \varphi_{g_{\II},A^*_{\II}}^{(r)}(x)  -
\Theta_{f_{\II}}(x; A^*_{\II}), \quad x > A^*_{\II}.
\end{align*}
For $x \leq A^*_{\II}$, we have $\widetilde{u}(x) = g(x)$.
\item If $A^*_{\II}=\infty$, $\widetilde{u}_{\II}(x) = g_{\II}(x)$ for any $x \in \R$ with optimal stopping time $\tau^* = 0$ a.s.
\item If $A^*_{\II}=-\infty$, then $\tau^*=\infty$ a.s.\ is optimal over $\S$ and \eqref{u_I} holds.
\end{enumerate}
\end{corollary}


\subsection{Two-stage problem}
 In order to gain intuition, we first consider the case with $M=2$ and obtain $U^{(2)}(x)$ and $\widetilde{U}^{(2)}(x)$ under Assumptions \ref{assump_h_multi} and  \ref{assump_h_multi2}, respectively.   Following the procedures discussed above, $A_{m}^* \in [-\infty, \infty]$, or the root of $\Lambda_m(A) = 0$, is well-defined for $m=1,2$.  As a special case of \eqref{def_g_f},
\begin{align}
f_2 \equiv F_2, \quad f_1 \equiv F_1 - F_2 \equiv F_1 - f_2, \quad \textrm{and} \quad
f_{\{1,2\}} \equiv F_1 \equiv f_1 + f_2. \label{F_relation_M_2}
\end{align} 
We shall consider the cases (i) $A_1^* > A_2^*$ and (ii)  $A_1^* \leq A_2^*$, separately.  For (i), we shall show that $(\tau_{A_1^*}, \tau_{A_2^*})$ is optimal.  For (ii), we shall show that simultaneous stoppings are optimal.
We first consider the former.
\begin{proposition} \label{lemma_monotone_case}
If $\infty \geq A_{1}^* > A_{2}^* \geq -\infty$, then $(\tau_{A_{1}^*}, \tau_{A_2^*})$ is optimal; the value function is given by $U^{(2)}(x) = \sum_{m=1,2} u_{A^*_{m}}^{\{m\}}(x)$ and $\widetilde{U}^{(2)}(x) = \sum_{m=1,2} u_{A^*_{m}}^{\{m\}}(x)$ under Assumptions \ref{assump_h_multi} and \ref{assump_h_multi2}, respectively.  In particular, under Assumption \ref{assump_h_multi} and if $\infty > A_{1}^* > A_{2}^* > -\infty$,
\begin{align}
\begin{split}
U^{(2)}(x) 
&= \sum_{m=1,2} \Big( K_m Z^{(r)} (x-A^*_m) - b_m \Big[ \overline{Z}^{(r)}(x-A_{m}^*) + \big(A_{m}^*- \frac {\psi'(0+)} r \big)Z^{(r)}(x-A_{m}^*) + \frac {\psi'(0+)} r \Big] 
\\& \qquad - \sum_{i=1}^{N_m} c_{mi} e^{a_{mi} x} Z_{a_{mi}}^{(r - \psi(a_{mi}))} (x-A^*_m) \Big) \\ &- \int_{A^*_1}^x W^{(r)}(x-y) F_1(y) \diff y - \int_{A^*_2}^{A^*_1} W^{(r)}(x-y) F_2(y) \diff y. 
\end{split}
\label{J_2}
\end{align}
\end{proposition}
\begin{proof}
Suppose Assumption \ref{assump_h_multi} holds. By relaxing the constraint that $\tau^{(1)} \leq \tau^{(2)}$, we can obtain an upper bound:
\begin{align*}
U^{(2)}(x) \leq \overline{U}^{(2)}(x) := \sum_{m=1,2} \sup_{\tau^{(m)} \in \S} u(x, \tau^{(m)}; f_m, g_m) = \sum_{m=1,2} u_{A^*_{m}}^{\{m\}}(x), 
\end{align*}
where the last equality holds by Corollary \ref{corollary_w}.  On the other hand, because $\tau_{A^*_1} \leq \tau_{A^*_2}$ a.s.\  (hence $(\tau_{A^*_1}, \tau_{A^*_2}) \in \S_2$) thanks to $A_1^*> A_2^*$,  we have  $U^{(2)}(x) \geq \overline{U}^{(2)}(x)$, as desired.  The same result holds under Assumption \ref{assump_h_multi2}  by relaxing the constraint that $A_1 \geq A_2$ and noticing that $(\tau_{A^*_1}, \tau_{A^*_2}) \in \ST_2$.

For the second claim, because $A_1^* > A_2^*$ and by \eqref{F_relation_M_2},
\begin{align*}
\sum_{m=1,2} \Theta_{f_m} (x, A_m^*) = \sum_{m=1,2}\int_{A^*_m}^x W^{(r)}(x-y) f_m(y) \diff y 
=\int_{A^*_1}^x W^{(r)}(x-y) F_1(y) \diff y + \int_{A^*_2}^{A_1^*} W^{(r)}(x-y) F_2(y) \diff y,
\end{align*}
and hence \eqref{J_2} holds in view of Corollary \ref{corollary_w}.

\end{proof}

Now consider the case  $-\infty \leq A_1^* \leq A_2^* \leq \infty$.
\begin{lemma} \label{lemma_region}
Suppose  $-\infty \leq A_1^* \leq A_2^* \leq \infty$.
Under Assumption \ref{assump_h_multi} (resp.\ Assumption \ref{assump_h_multi2}),  the first optimal stopping cannot occur on $(A_2^*, \infty)$; namely if $\tau^{*(1)}$ is the optimal first stopping time in the sense of \eqref{multi_stage_objective} (resp.\ \eqref{multi_stage_objective_tilde}), then $X_{\tau^{*(1)}} \in (-\infty, A_2^*]$ a.s. on $\{ \tau^{*(1)} < \infty \}$.
\end{lemma}
\begin{proof}
The result is immediate when $A_2^* = \infty$ and hence we assume $A_2^* < \infty$.

Suppose Assumption \ref{assump_h_multi} holds, and in order to derive a contradiction, we suppose there exists some $\hat{x} > A_2^*$ at which it is optimal to stop.  Under this assumption, the value function must satisfy
\begin{align}
 U^{(2)}(\hat{x}) = g_1(\hat{x}) +  \sup_{\tau \in \S} u(\hat{x}, \tau; f_2, g_2) = g_1(\hat{x}) +  u_{A_2^*}^{\{2\}}(\hat{x}). \label{J_contradiction}
\end{align}
We shall show that this is in fact smaller than $u_{A_2^*}^{\{1\}}(\hat{x}) + u_{A_2^*}^{\{2\}}(\hat{x})$, which is the value obtained by $(\tau^{(1)}, \tau^{(2)}) = (\tau_{A_2^*}, \tau_{A_2^*}) \in \ST_2 \subset  \S_2$. By \eqref{cond_continuous_fit} and \eqref{u_A_derivative},
\begin{align}
&\Lambda_1(A)\leq (\geq) 0 \Longrightarrow \frac \partial {\partial A} u_A^{\{1\}}(\hat{x})  \geq (\leq) 0, \quad  \forall A < \hat{x}, \label{equivalence_derivative_Phi} \\
&u_A^{\{1\}}(A+) = g_1(A) +  W^{(r)}(0) \Lambda_1(A). \label{cont_fit_difference}
\end{align}
By \eqref{equivalence_derivative_Phi}-\eqref{cont_fit_difference} and because $A_1^* \leq A_2^* < \hat{x}$ and $\Lambda_1$ is increasing,
\begin{align*}
u_{A^*_2}^{\{1\}}(\hat{x}) > \lim_{A \uparrow \hat{x}}u_A^{\{1\}}(\hat{x}) = g_1(\hat{x}) + W^{(r)}(0) \Lambda_1(\hat{x}) \geq g_1(\hat{x}).
\end{align*}
Regarding the last inequality, for the unbounded variation case, it holds because $W^{(r)}(0)=0$ by Remark \ref{remark_smoothness_zero} (2).  For the bounded variation case, it also holds because \eqref{equivalence_derivative_Phi} and $\hat{x} > A_1^*$ imply $\Lambda_1(\hat{x}) > 0$.  Therefore, we get, by \eqref{J_contradiction}, $U^{(2)}(\hat{x}) < u_{A_2^*}^{\{1\}}(\hat{x}) + u_{A_2^*}^{\{2\}}(\hat{x})$ leading to a contradiction.  Because $\hat{x}$ is arbitrary on $(A_2^*, \infty)$, we have the claim.  The same contradiction can be derived under Assumption \ref{assump_h_multi2} because $(\tau_{A_2^*},\tau_{A_2^*}) \in \ST_2$.
\end{proof}

The following proposition suggests under $-\infty \leq A_1^* \leq A_2^* \leq \infty$ that the optimal strategy is the simultaneous stoppings corresponding to the threshold level $A^*_{\{1,2\}}$, which is the value that makes $\Lambda_{\{1,2\}} \equiv \Lambda_1 + \Lambda_2$ as in \eqref{Lambda_II} vanish.

\begin{proposition} \label{case_two_second} Suppose $-\infty \leq A_1^* \leq A_2^* \leq \infty$. 
\begin{enumerate}
\item We have $A_1^* \leq  A^*_{\{1,2\}} \leq A_2^*$.
\item It is optimal to stop simultaneously and the value function is given by $U^{(2)}(x) = u_{A^*_{\{1,2\}}}^{\{1,2\}}(x)$ under Assumption  \ref{assump_h_multi} and $\widetilde{U}^{(2)}(x) = u_{A^*_{\{1,2\}}}^{\{1,2\}}(x)$ under Assumption  \ref{assump_h_multi2}.
\end{enumerate}
\end{proposition}
\begin{proof}
(1)   Because both $\Lambda_1$ and $\Lambda_2$ are increasing, $\Lambda_{\{1,2\}}$ is increasing as well.  Because $A_1^* \leq A_2^*$, we have $\Lambda_1(A_2^*) \geq 0$ and hence $\Lambda_{\{1,2\}}(A_2^*) \geq 0$.  Similarly, $\Lambda_{\{1,2\}}(A_1^*) \leq 0$.   The increasing property of $\Lambda_{\{1,2\}}$ now shows the claim.

(2)  
Under Assumption  \ref{assump_h_multi},  for any pair of stopping times $(\tau^{(1)}, \tau^{(2)}) \in \S_2$, because $F_2 = f_2$ as in \eqref{F_relation_M_2} and by the strong Markov property of the \lev process $X$,
\begin{align} \label{inequality_sup}
\begin{split}
&\E^x \Big[ \sum_{m=1,2} e^{-r \tau^{(m)}}
g_m(X_{\tau^{(m)}}) 1_{\{\tau^{(m)} < \infty \}} + \int_0^{\tau^{(1)}} e^{-rt}  F_1(X_t) \diff t +  \int_{\tau^{(1)}}^{\tau^{(2)}}e^{-rt} F_2(X_t)  \diff
t \Big]\\
&\leq \E^x \Big[  e^{-r \tau^{(1)}}
g_1(X_{\tau^{(1)}}) 1_{\{\tau^{(1)} < \infty \}} + \int_0^{\tau^{(1)}} e^{-rt}  F_1(X_t) \diff t   +  1_{\{ \tau^{(1)} < \infty \}} e^{- r \tau^{(1)}} \sup_{\tau \in \S}u(X_{\tau^{(1)}}, \tau; f_2, g_2)  \Big]\\
&= u(x, \tau^{(1)}; F_1, g_1 + u_{A^*_2}^{\{2\}}).
\end{split}
\end{align}

Similarly, under Assumption  \ref{assump_h_multi2}, \eqref{inequality_sup} also holds for any $(\tau^{(1)}, \tau^{(2)}) \in \ST_2$ by replacing $\S$ with $\ST$ in the expectation of the third term.
This together with Lemma \ref{lemma_region} shows that $U^{(2)}(x)$ (resp.\ $\widetilde{U}^{(2)}(x)$) is less than or equal to
\begin{align}
\sup_{\tau \in \mathcal{S}(A_2^*)} u(x, \tau; F_1, g_1 + u_{A^*_2}^{\{2\}}) \label{exp_sup_A2}
\end{align}
under Assumption  \ref{assump_h_multi} (resp.\ Assumption  \ref{assump_h_multi2})
where $\mathcal{S}(A_2^*)$ is the set of $\tau \in \S$ (resp.\ $\tau \in \ST$) such that $X_{\tau} \in (-\infty, A_2^*]$ a.s.\ on $\{\tau < \infty\}$.  Because, for any $\tau \in \mathcal{S}(A_2^*)$, $u_{A^*_2}^{\{2\}} (X_\tau) = g_2(X_{\tau})$ a.s.\ on  $\{\tau < \infty\}$ and because $F_1  = f_{\{1,2\}}$ as in \eqref{F_relation_M_2},
 \eqref{exp_sup_A2} equals
\begin{align*}
\sup_{\tau \in \mathcal{S}(A_2^*)} u(x, \tau; f_{\{1,2\}}, g_{\{1,2\}}) 
\leq \sup_{\tau \in \mathcal{S}} u(x, \tau; f_{\{1,2\}}, g_{\{1,2\}})  = u_{A^*_{\{1,2\}}}^{\{1,2\}}(x),
\end{align*}
by Corollary \ref{corollary_w} and because $\S(A_2^*) \subset \S$. Namely, $U^{(2)}(x) \leq u_{A^*_{\{1,2\}}}^{\{1,2\}}(x)$ under Assumption \ref{assump_h_multi} and $\widetilde{U}^{(2)}(x) \leq u_{A^*_{\{1,2\}}}^{\{1,2\}}(x)$ under Assumption \ref{assump_h_multi2}. These in fact hold with equality because $u_{A^*_{\{1,2\}}}^{\{1,2\}}(x)$ is attained by $(\tau_{A^*_{\{1,2\}}}, \tau_{A^*_{\{1,2\}}}) \in \ST_2 \subset \S_2$.

\end{proof}

\subsection{Multiple-stage problem} \label{subsection_multi_stage}
We now generalize the results to the multiple-stage case and solve \eqref{multi_stage_objective}-\eqref{multi_stage_objective_tilde} or equivalently \eqref{multi_stage_objective2}-\eqref{multi_stage_objective2_tilde} with $M \geq 3$.  For $1 \leq m \leq M$, let
\begin{align}
U_m^{(M)}(x) := \sup_{(\tau^{(m)}, \ldots, \tau^{(M)}) \in \S_{M-m+1}} \sum_{k=m}^M u(x, \tau^{(k)}; f_k, g_k), \label{multi_stage_objective_M} \\
\widetilde{U}_m^{(M)}(x) := \sup_{(\tau^{(m)} , \ldots,\tau^{(M)}) \in \ST_{M-m+1}} \sum_{k=m}^M u(x, \tau^{(k)}; f_k, g_k). \label{multi_stage_objective_M_tilde}
\end{align}
In particular, $U^{(M)} \equiv U_1^{(M)}$ and $\widetilde{U}^{(M)} \equiv \widetilde{U}_1^{(M)}$, and by Corollary \ref{corollary_w}
\begin{align}
U_M^{(M)}(x) = \sup_{\tau \in \S} u(x, \tau; f_M, g_M)  = u^{\{M\}}_{A^*_M}(x) \quad \textrm{and} \quad \widetilde{U}_M^{(M)}(x) = \sup_{\tau \in \ST} u(x, \tau; f_M, g_M)  = u^{\{M\}}_{A^*_M}(x), \label{multi_stage_objective_M_last}
\end{align}
under Assumptions  \ref{assump_h_multi} and  \ref{assump_h_multi2}, respectively.
The expressions for $U_{M-1}^{(M)}$ and $\widetilde{U}_{M-1}^{(M)}$ can also be obtained as in the two-stage case.

Given $1 \leq m \leq M$, let us partition $\{ m, m+1, \ldots, M \}$ to an $L(m)$ number of (non-empty) disjoint sets $\II_m := \{ \II(k; m), 1 \leq k \leq L(m) \}$ such that
\begin{align*}
\{ m, m+1, \ldots, M \} = \II(1; m)  \cup  \cdots \cup  \II(L(m); m)
\end{align*}
where, if $L(m)=1$, $\II(1;m)=\{ m, \ldots, M\}$ and, if $L(m) \geq 2$,
\begin{align*}
\II(1;m) &:= \{ m, \ldots, n_{1,m}-1  \}, \\
\II(l; m) &:= \{ n_{l-1,m}, \ldots, n_{l,m}-1  \}, \quad 2 \leq l \leq L(m)-1, \\
\II(L(m); m) &:= \{ n_{L(m)-1,m}, \ldots, M  \},
\end{align*}
for some integers $m < n_{1,m} < \cdots < n_{L(m)-1,m} < M$.  We
consider the strategy such that, if $k$ and $l$ are in the same set, then the $k$-th and $l$-th stops occur simultaneously a.s.

We shall show that \eqref{multi_stage_objective_M} and \eqref{multi_stage_objective_M_tilde}, for any $1 \leq m \leq M$, can be solved by  a strategy with some partition $\II_m^* := \{ \II^*(k; m), 1 \leq k \leq L^*(m) \}$ satisfying 
\begin{align*}
A^*_{\II^*(1;m)} > \cdots > A^*_{\II^*(L^*(m); m)},
\end{align*}
where $A^*_{\II}$ is defined as in \eqref{Lambda_II} for any set $\II$.  The corresponding expected value becomes
\begin{align} \label{J_I}
\begin{split}
U_{m, \II_m^*}^{(M)} (x) :=  \sum_{k=1}^{L^*(m)} u(x, \tau_{A^*_{\II^*(k; m)}}; f_{\II^*(k;m)},  g_{\II^*(k;m)}) 
= \sum_{k=1}^{L^*(m)}  u_{A^*_{\II^*(k;m)}}^{\II^*(k;m)}(x),
\end{split}
\end{align}
whose strategy is given by
for any $m \leq n \leq M$,
\begin{align*}
\tau^{*(n)} &= \tau_{A^*_{\II^*(k;m)}} \quad \textrm{ for the unique $1 \leq k \leq L^*(m)$ such that $n \in \II^*(k;m)$}.
\end{align*}

We shall show that \eqref{J_I} is optimal, i.e.\ $U_m^{(M)} = U_{m, \II_m^*}^{(M)}$ under Assumption \ref{assump_h_multi} and $\widetilde{U}_m^{(M)} = U_{m, \II_m^*}^{(M)}$ under Assumption \ref{assump_h_multi2} for any $1 \leq m \leq M$.  Moreover, $\II_m^*$ can be obtained inductively moving backwards starting from $\II^*_M$ such that $L^*(M) = 1$ and $\II^*(1; M) = \{ M \}$.
For the inductive step, the following algorithm outputs $\II^*_{m-1}$ from $\II^*_{m}$ for any $2 \leq m \leq M$.  By repeating this, we can obtain the partition $\II_1^*$; the resulting $U_{1, \II^*_1}^{(M)}$ as in \eqref{J_I} becomes the value function $U^{(M)}=U^{(M)}_1$ ($\widetilde{U}^{(M)}=U^{(M)}_1$).

%
%
%
%

\begin{center}
\line(1,0){400}
\end{center}

 \textbf{Algorithm $\II^*_{m-1}$ = Update($\II^*_{m},m$)}
\begin{description}
\item[Step 1] Set $i=1$.
\item[Step 2]  Set 
\begin{align*}
\hat{\II} := \left\{ \begin{array}{ll} \{m-1\}, & i = 1, \\ \{m-1\} \cup \II^*(1;m) \cup \cdots \cup \II^*(i-1;m), & i \geq 2. \end{array} \right. 
\end{align*}
\item[Step 3]Compute $A^*_{\hat{\II}}$ and
\begin{enumerate}
\item if $i = L^*(m)+1$, then \textbf{stop and return} $\II^*_{m-1} = \{ \II^*(1; m-1) \}$ with $L^*(m-1) = 1$ and $\II^*(1; m-1) = \{m-1, \ldots, M\}$;
\item if $A^*_{\hat{\II}} > A^*_{\II^*(i; m)}$, then \textbf{stop and return} $\II^*_{m-1} = \{ \II^*(k; m-1), 1 \leq k \leq L^*(m-1)\}$ with $L^*(m-1) = L^*(m) - i+2$ and
\begin{align} \label{I_update}
\II^*(1; m-1) = \hat{\II} \quad \textrm{and} \quad \II^*(l; m-1) = \II^*(l+i-2; m), \quad 2 \leq l \leq L^*(m-1);
\end{align}

\item if $A^*_{\hat{\II}}\leq A^*_{\II^*(i;m)}$, set $i = i+1$ and go back to \textbf{Step 2}.
\end{enumerate}
\end{description}
\begin{center}
\line(1,0){400}
\end{center}

The role of the algorithm is in words to extend from $n(=M-m+1)$-stage problem to $n+1(=M-m+2)$-stage problem.  The idea is similar to what we discussed in the previous section on how to extend from a one-stage problem to a two-stage problem.  When a new initial stage is added, the corresponding threshold value $A^*_{\hat{\II}}$ is first calculated.  Depending on whether its value is higher than that of the subsequent stages or not, simultaneous stoppings may become optimal.  For $n$ larger than two, we must solve it recursively by keeping updating the set $\hat{\II}$, or the set of the first (simultaneous) stoppings, as given in this algorithm.   If $A^*_{\hat{\II}}$ is low, the strategy of the new initial stage may naturally depend on the strategies of all the subsequent stages.  Unlike the extension to the two-stage problem which only needs to take into account the strategy of the  stage immediately next, it needs to reflect the strategies of all subsequent stages.

We prove the following under Assumption \ref{assump_h_multi} for the optimality \eqref{multi_stage_objective}.  As is already clear after the detailed discussion on the two-stage case, only a slight modification is needed for \eqref{multi_stage_objective_tilde} under Assumption \ref{assump_h_multi2}.

\begin{lemma}  \label{lemma_inductive_step}
In view of the algorithm above, suppose Assumption \ref{assump_h_multi} and fix $2 \leq m \leq M$. Given that $\II^*_m$ satisfies, for every $1 \leq l \leq L^*(m)$,
\begin{align}  \label{J_min_inductive_base}
U_{\min \mathcal{I}^*(l; m)}^{(M)}(x) = \sum_{k=l}^{L^*(m)} u_{A^*_{\II^*(k; m)}}^{\II^*(k; m)}(x),
\end{align}
and is used as an input in the algorithm.  Then, 
we have the following.
\begin{enumerate}
\item At the end of \textbf{Step 2}, if $1 \leq i \leq L^*(m)$,  
\begin{align}
U_{m-1}^{(M)}(x)  \leq\sup_{\tau \in \S}  \Big[ u(x, \tau; f_{\hat{\II}}, g_{\hat{\II}}) +  \sum_{k=i}^{L^*(m)} u(x, \vartheta_\tau(A^*_{\II^*(k; m)}); f_{\II^*(k; m)}, g_{\II^*(k; m)})  \Big] 
\label{J_m-1_middle}
\end{align}
where $\vartheta_\nu(A) := \nu + \tau_A \circ \theta_\nu$ for any $\nu \in \S$ and $A \in \R$ with the time-shift operator $\theta_t$, 
and if $i = L^*(m)+1$
\begin{align*}
U_{m-1}^{(M)}(x)  =\sup_{\tau \in \S}  u(x, \tau; f_{\hat{\II}},  g_{\hat{\II}}) = u^{\{m-1,\ldots, M\}}_{A^*_{\{m-1,\ldots,M\}}}(x).
\end{align*}
\item Let $\II^*_{m-1}$ be produced by the algorithm.  For any $1 \leq l \leq L^*(m-1)$,
\begin{align} \label{J_min_inductive}
U_{\min \mathcal{I}^*(l; m-1)}^{(M)}(x)   = \sum_{k=l}^{L^*(m-1)} u_{A^*_{\II^*(k; m-1)}}^{\II^*(k; m-1)}(x).  
\end{align}
\end{enumerate}
\end{lemma}
\begin{proof}
(1) We shall proceed by mathematical induction.

(Base-step) Suppose $i = 1$.  By our assumption \eqref{J_min_inductive_base} and by an argument similar to \eqref{inequality_sup}, 
\begin{align*}
U_{m-1}^{(M)}(x)  &\leq \sup_{\tau \in \S}  \E^x \Big[ e^{-r \tau}
g_{m-1} (X_{\tau})  1_{\{\tau < \infty \}}  +   \int_0^{\tau} e^{-rt}  F_{m-1}(X_t) \diff t + e^{-r \tau} U_m^{(M)}(X_\tau)  1_{\{ \tau < \infty \}} \Big] \\
&= \sup_{\tau \in \S}  \Big[ u(x, \tau; f_{m-1}, g_{m-1}) +  \sum_{k=1}^{L^*(m)} u(x, \vartheta_\tau(A^*_{\II^*(k; m)}); f_{\II^*(k; m)}, g_{\II^*(k; m)})  \Big]. 
\end{align*}
Now  for $i=1$ \eqref{J_m-1_middle} holds because $\hat{\II} = \{m-1\}$; this becomes the base case.

(Inductive-step) Now we assume  \eqref{J_m-1_middle} for $i=j \leq L^*(m)-1$, i.e., $\hat{\II} = \{m-1\} \cup \II^*(1;m) \cup \cdots \cup \II^*(j-1;m)$ and
\begin{align}
U_{m-1}^{(M)}(x)  \leq\sup_{\tau \in \S}  \Big[ u(x, \tau; f_{\hat{\II}}, g_{\hat{\II}}) +  \sum_{k=j}^{L^*(m)} u(x, \vartheta_\tau(A^*_{\II^*(k; m)}); f_{\II^*(k; m)}, g_{\II^*(k; m)})  \Big], \label{J_m-1_middle_j}
\end{align}
 and show that it will hold for $i=j+1$.



Because when $A^*_{\hat{\II}}  > A^*_{\mathcal{I}^*(j;m)}$ the algorithm stops at $j$ and never returns to \textbf{Step 2}, we suppose here that $A^*_{\hat{\II}}  \leq A^*_{\mathcal{I}^*(j;m)}$.  
In view of the right-hand side of \eqref{J_m-1_middle_j}, if there exists some $\hat{x} > A^*_{\mathcal{I}^*(j;m)}$ at which it is optimal to stop, then the value function becomes $g_{\hat{\II}}(\hat{x}) + U_{\min \II^*(j;m)}^{(M)}(\hat{x})$ by our assumption \eqref{J_min_inductive_base}.  Using the same reasoning as in Lemma \ref{lemma_region}, this is in fact smaller than $u_{A^*_{\II^*(j;m)}}^{\hat{\II}}(\hat{x})+ U_{\min \II^*(j;m)}^{(M)}(\hat{x})$. 
 Hence it is never optimal to stop on $(A^*_{\mathcal{I}^*(j;m)}, \infty)$ for the optimization problem on the right-hand side of \eqref{J_m-1_middle_j} (see also the proof of Proposition \ref{case_two_second}).



Now let $\mathcal{S}(A^*_{\II^*(j; m)})$ be the set of all stopping times at which $X \in (-\infty, A^*_{\II^*(j; m)}]$ a.s.
For all $\tau \in \mathcal{S}(A^*_{\II^*(j; m)})$, we have $\tau = \vartheta_\tau (A^*_{\II^*(j;m)}) $ a.s.\ and hence $u(x, \vartheta_\tau(A^*_{\II^*(j; m)}); f_{\II^*(j; m)}, g_{\II^*(j; m)}) = u(x, \tau; f_{\II^*(j; m)}, g_{\II^*(j; m)})$.  Therefore \eqref{J_m-1_middle_j} implies
\begin{align*}
U_{m-1}^{(M)}(x)  &\leq\sup_{\tau \in \mathcal{S}(A^*_{\II^*(j; m)})}  \Big[ u(x, \tau; f_{\hat{\II} \cup \II^*(j; m)}, g_{\hat{\II} \cup \II^*(j; m)}) +  \sum_{k=j+1}^{L^*(m)} u(x, \vartheta_\tau(A^*_{\II^*(k; m)}); f_{\II^*(k; m)}, g_{\II^*(k; m)})  \Big]  \\
&\leq\sup_{\tau \in \mathcal{S}}  \Big[ u(x, \tau; f_{\hat{\II} \cup \II^*(j; m)}, g_{\hat{\II} \cup \II^*(j; m)}) +  \sum_{k=j+1}^{L^*(m)} u(x, \vartheta_\tau(A^*_{\II^*(k; m)}); f_{\II^*(k; m)}, g_{\II^*(k; m)})  \Big]. 
\end{align*}
Hence, \eqref{J_m-1_middle} holds for $i=j+1$, as desired. This proves (1) by mathematical induction.

(2) When the algorithm stops,  it is either (i) $i = L^*(m)+1$ at \textbf{Step 3}(1) or (ii) $A^*_{\hat{\II}} > A^*_{\mathcal{I}^*(i;m)}$ at \textbf{Step 3}(2).  \\
(i) Suppose $i = L^*(m)+1$.  In this case, $\hat{\II} = \{m-1, \ldots, M\}$ and, by \eqref{J_m-1_middle}, 
\begin{align*}
U_{\min \mathcal{I}^*(1; m-1)}^{(M)}(x)  = U_{m-1}^{(M)} (x) \leq \sup_{\tau \in \S} u(x, \tau; f_{\{ m-1, \ldots, M\}}, g_{\{ m-1, \ldots, M\}}) = u^{\{m-1,\ldots, M\}}_{A^*_{\{m-1,\ldots,M\}}}(x),
\end{align*}
which in fact holds by equality because the right-hand side is attained by $(\tau_{A^*_{\{m-1,\ldots,M\}}}, \ldots, \tau_{A^*_{\{m-1,\ldots,M\}}}) \in \ST_{M-m+2} \subset \S_{M-m+2}$. \\
(ii)  Suppose the algorithm exits at $i$ with $A^*_{\hat{\II}}> A^*_{\mathcal{I}^*(i;m)}$.  
By  \eqref{J_m-1_middle}, we have
\begin{align*}
U_{m-1}^{(M)}(x)  
&\leq  \sup_{\tau \in \S}  u(x,\tau; f_{\hat{\II}},g_{\hat{\II}})  + \sup_{\tau \in \S}\sum_{k=i}^{L^*(m)} u(x, \vartheta_\tau(A^*_{\II^*(k; m)});  f_{\II^*(k;m)}, g_{\II^*(k,m)}).
\end{align*}
Regarding the second supremum of the right-hand side, the strategy $\{ \tau^{(l)}; \min \II^*(i,m) \leq l \leq M\}$, defined by $\tau^{(l)} = \vartheta_\tau(A^*_{\II^*(k; m)})$ for the unique $i \leq k \leq L^*(m)$ such that  $l \in \II^*(k;m)$, is feasible (or in $\S_{M + 1 - \min \II^*(i,m)}$) for any stopping time $\tau \in \S$ and therefore
\begin{align*}
&\sup_{\tau \in \S}\sum_{k=i}^{L^*(m)} u(x, \vartheta_\tau(A^*_{\II^*(k; m)});  f_{\II^*(k;m)}, g_{\II^*(k,m)}) \\ &\leq
 \sup_{(\tau^{(\min \II^*(i,m))}, \ldots, \tau^{(M)}) \in \S_{M + 1 - \min \II^*(i,m)} \;} \sum_{k=\min \II^*(i; m)}^M u(x, \tau^{(k)}; f_k, g_k)  = U_{\min \II^*(i; m)}^{(M)}(x). 
\end{align*}
Hence, we obtain a bound $U_{m-1}^{(M)}(x)  
\leq u^{\hat{\II}}_{A^*_{\hat{\II}}} (x) + U_{\min \II^*(i;m)}^{(M)}(x)$.  This together with \eqref{I_update} and \eqref{J_min_inductive_base}  shows
\begin{align*}
U_{m-1}^{(M)}(x)  
\leq u^{\hat{\II}}_{A^*_{\hat{\II}}} (x) + \sum_{k=i}^{L^*(m)} u_{A^*_{\II^*(k; m)}}^{\II^*(k; m)}(x) =  \sum_{k=1}^{L^*(m-1)} u_{A^*_{\II^*(k; m-1)}}^{\II^*(k; m-1)}(x).
\end{align*}
This holds by equality because the right-hand side is attained by a feasible strategy defined by $\II^*_{m-1}$.
%
%
This shows \eqref{J_min_inductive} for case $l=1$.

On the other hand, for any $2 \leq l \leq L^*(m-1)$, by \eqref{I_update} and \eqref{J_min_inductive_base},
\begin{align*} 
&U_{\min \mathcal{I}^*(l; m-1)}^{(M)}(x)  = U_{\min \mathcal{I}^*(l+i-2; m)}^{(M)}(x)  = \sum_{k=l+i-2}^{L^*(m)} u_{A^*_{\II^*(k;m)}}^{\II^*(k;m)}(x) = \sum_{k=l}^{L^*(m-1)} u_{A^*_{\II^*(k;m-1)}}^{\II^*(k;m-1)} (x),
\end{align*}
which guarantees \eqref{J_min_inductive}, as desired.

\end{proof}

Using Lemma \ref{lemma_inductive_step} as an inductive step, the main theorem is immediate.  Indeed, \eqref{J_min_inductive_base} holds trivially for $M$ by Corollary \ref{corollary_w}.  By applying the algorithm $M-1$ times, we can obtain \eqref{corollary_w} for $M-1,M-2,\ldots, 1$.

\begin{theorem} \label{theorem_multiple}
 Let $\{ \II^*_m; 1 \leq m \leq M\}$ be produced by the algorithm.  
\begin{enumerate}
\item Under Assumption \ref{assump_h_multi}, for every $1 \leq m \leq M$ and $1 \leq i \leq L^*(m)$,
\begin{align*} 
U_{\min \II^*(i; m)}^{(M)} (x) =  \sum_{k=i}^{L^*(m)} u_{A^*_{\II^*(k;m)}}^{\II^*(k;m)} (x). 
\end{align*}
In particular,
\begin{align} \label{result_multiple1}
U^{(M)}(x) \equiv U_{\min \II^*(1;1)}^{(M)} (x) =  \sum_{k=1}^{L^*(1)} u_{A^*_{\II^*(k;1)}}^{\II^*(k;1)} (x). 
\end{align}
\item Under Assumption \ref{assump_h_multi2}, for every $1 \leq m \leq M$ and $1 \leq i \leq L^*(m)$,
\begin{align*} 
\widetilde{U}_{\min \II^*(i; m)}^{(M)} (x) =  \sum_{k=i}^{L^*(m)} u_{A^*_{\II^*(k;m)}}^{\II^*(k;m)} (x). 
\end{align*}
In particular,
\begin{align} \label{result_multiple2}
\widetilde{U}^{(M)}(x) \equiv \widetilde{U}_{\min \II^*(1;1)}^{(M)} (x) =  \sum_{k=1}^{L^*(1)} u_{A^*_{\II^*(k;1)}}^{\II^*(k;1)} (x). 
\end{align}
\end{enumerate}
\end{theorem}

\section{phase-type case and numerical examples}\label{section_numer}

In this section, we consider   spectrally negative \lev processes with i.i.d.\ phase-type jumps and provide numerical examples.  Any \lev process can be approximated  by those with phase-type jumps (phase-type \lev processes); see, e.g., \cite{Feldmann_1998}.  In a related work, Egami and Yamazaki \cite{Egami_Yamazaki_2010_2} approximate the scale function of the spectrally negative \lev process by those of phase-type \lev processes.

\subsection{Spectrally negative \lev processes with phase-type jumps}
Let $X$ be a spectrally negative \lev process of the form
\begin{equation}
  X_t  - X_0=\mu t+\sigma B_t - \sum_{n=1}^{N_t} Z_n, \quad 0\le t <\infty. \label{lev_hyperexp}
\end{equation}
Here $B=\{B_t; t\ge 0\}$ is a standard Brownian motion, $N=\{N_t; t\ge 0\}$ is a Poisson process with arrival rate $\lambda$, and  $Z = \left\{ Z_n; n = 1,2,\ldots \right\}$ is an i.i.d.\ sequence of phase-type-distributed random variables with representation $(m,{\bm \alpha},{\bm T})$; see \cite{Asmussen_2004}.
These processes are assumed mutually independent. The Laplace exponent  \eqref{laplace_exp} of $X$ is then
\begin{align*}
 \psi(s)   = \mu s + \frac 1 2 \sigma^2 s^2 + \kappa \left( {\bm \alpha} (s {\bm I} - {\bm{T}})^{-1} {\bm t} -1 \right), \quad \textrm{with } {\bm t} = - {\bm T} 
[1,\ldots,1]',
 \end{align*}
which can be extended to $s \in \mathbb{C}$ except at the eigenvalues of ${\bm T}$.    Suppose $\{ -\xi_{i,r}; i \in \mathcal{I}_r\}$ is the set of the roots of the equality $\psi(s) = r$ with negative real parts, and if these are assumed distinct and $\sigma > 0$, then
the scale function can be written
\begin{align*}
W^{(r)}(x) =  \sum_{i \in \mathcal{I}_r}C_{i}  (e^{\Phi_r x}- e^{-\xi_{i,r}x}),
\end{align*}
where
\begin{align*}
C_{i} &:= \left. \frac { s+\xi_{i,r}} {r-\psi(s)} \right|_{s = -\xi_{i,r}} = - \frac 1 {\psi'(-\xi_{i,r})};
\end{align*}
see \cite{Egami_Yamazaki_2010_2}.  Here $\{ \xi_{i,r}; i \in \mathcal{I}_r \}$ and  $\{ C_{i}; i \in \mathcal{I}_r \}$ are possibly complex-valued.

With $\Phi_r^{(c)} := \lapinv - c$ and $\xi_{i,r}^{(c)} := \xi_{i,r} + c$,  for any $i \in \mathcal{I}_r$ and  $c \geq 0$, we have by \eqref{scale_measure_change} for any $x \geq 0$
\begin{align} \label{scale_phase-type_c}
\begin{split}
W_c^{(r-\psi(c))}(x) &=  \sum_{ i \in \mathcal{I}_r} C_i \Big[ e^{\Phi_r^{(c)} x} - e^{-\xi_{i,r}^{(c)}x} \Big], \\
Z^{(r-\psi(c))}_c(x) &=  1 + (r - \psi(c)) \sum_{ i \in \mathcal{I}_r} C_i \Big[ \frac 1 {\Phi_r^{(c)}} ( e^{\Phi_r^{(c)} x} - 1 ) + \frac 1 {\xi_{i,r}^{(c)}} ( e^{-\xi_{i,r}^{(c)}x} - 1 ) \Big].
\end{split}
\end{align}
Thanks to their forms as sums of exponential functions, the value function can be obtained analytically.

For our examples,  we assume $\sigma =0.2$ and $\mu = \lambda = 1$.  For the phase-type distribution for $Z$,   we assume $m = 6$ and
\begin{align*}
&{\bm T} = \left[ \begin{array}{rrrrrr}  -5.5209 &   0.0000 &   0.0000 &   0.0000 & 0.0000 &   0.0000\\
    0.0073 &  -5.4523 &   5.4443 &   0.0000 &  0.0000 &   0.0000 \\
5.4959 &   0.0000 &  -5.4959 &   0.0000 & 0.0000 &   0.0000\\
   0.2193 &  0.0030 &  0.2920 &  -5.6885 &  5.1589 &   0.0154 \\
0.2703 &   0.8484 &   0.0027 &   0.0000 &   -5.6502 &   4.5262\\
0.0020  & 4.8467 &  0.0157  & 0.0000  &0.0000 & -5.9780 \end{array} \right], \quad {\bm \alpha} =    \left[ \begin{array}{l}             0.0000 \\   0.0048 \\   0.0044 \\   0.9906 \\  0.0002 \\   0.0000 \end{array} \right],
\end{align*}
which give an approximation of the Weibull distribution with density function $f(x) = 2 x \exp \left\{ - x^2\right\}$, $x \geq 0$ (which satisfies Assumption \ref{assump_tail}),
obtained using the EM-algorithm; see \cite{Egami_Yamazaki_2010_2} regarding the approximation performance of the corresponding scale function.

\subsection{Numerical results on the one-stage problem}  We first consider the one-stage problem as studied in Section \ref{section_single_stopping}.     
In our numerical examples, we consider two examples for $g$ satisfying \eqref{g_polynomial}:
\begin{enumerate}
\item[(a)] \emph{mixture of exponential functions:} $g^{(exp)}= K- \sum_{i=1}^N c_i e^{a_i x}$
for some constants $K \in \R$ and $c_i, a_i > 0$, $1 \leq i \leq N$, $N \geq 0$;
\item[(b)] \emph{linear function}: $g^{(lin)}(x) :=-\alpha x$, $x \in \R$,
for some $\alpha > 0$.
\end{enumerate}
  Regarding $f$, we consider the following three examples:
\begin{enumerate}
\item[(i)] \emph{simple function}: $f^{(sim)}(y) := \sum_{-\infty < n < \infty} f^{(n)} 1_{I_n} (y)$ for some constants $\cdots < f^{(-2)} < f^{(-1)} < f^{(0)} < f^{(1)} < f^{(2)} < \cdots$ such that $-\infty < \lim_{n \downarrow -\infty} f^{(n)} \leq \lim_{n \uparrow \infty} f^{(n)} < \infty$  and subdivisions $I_n := (l_n, l_{n+1}]$ of $\R$;
\item[(ii)] \emph{linear function}: $f^{(lin)}(y) := b_1 (y + b_2)$ for some $b_1 > 0$ and $b_2 \in \R$;
\item[(iii)] \emph{exponential function with an upper bound}: $f^{(exp)}(y) := e^{(L y) \wedge B}$ for some $L > 0$ and $B \in \R$.
\end{enumerate}
These satisfy Assumption \ref{assump_optimality2}(1) and in particular (ii) and (iii) satisfy Assumption \ref{assump_optimality}(1).  Hence Proposition \ref{proposition_general} holds  (or $\widetilde{u} = u_{A^*}$)  for any choice and in particular Proposition \ref{theorem_polynomial} holds (or $u = \widetilde{u} = u_{A^*}$) for (ii) or (iii).

In order to implement the optimal strategy, we first obtain $A^*$ using \eqref{lambda_A_polynomial} and then compute the value function via \eqref{value_function}.  
In our numerical results, for  $g$ and $f$, we consider any combination of the following:
\begin{enumerate}
\item[(a)] $g = g^{(exp)}$ with $a = [0.1,0.2,0.3,0.4]$ and $c=[4,3,2,1]$ and $K = 10$;
\item[(b)]  $f = g^{(lin)}$ with  $\alpha =1$;
\end{enumerate}
and
\begin{enumerate}
\item[(i)] $f = \gamma f^{(sim)}$ with $I_1 = (-\infty, 0)$, $I_2 = [0,\infty)$, $f^{(1)}=-10$ and  $f^{(2)}=10$;
\item[(ii)]  $f = \gamma f^{(lin)}$ with  $b_1=1$ and  $b_2=0$;
\item[(iii)]  $f = \gamma f^{(exp)}$ with  $L=B=1$;
\end{enumerate}
for the weight parameter $\gamma =0,0.05, 0.1$.  

The results for (a) $g = g^{(exp)}$ and (b) $g = g^{(lin)}$ are graphically shown in Figures \ref{figure_exp} and  \ref{figure_lin}, respectively.  In each figure, we plot the function $\Lambda(\cdot)$ as in \eqref{Large_lambda} and the value function $u_{A^*}$ for each choice of $f$.  As can be confirmed, the function $\Lambda(\cdot)$ is indeed monotonically increasing and hence the unique root $A^*$ of  $\Lambda(A)=0$ can be obtained easily by the bisection method.  Using these optimal threshold levels, the value functions are computed via \eqref{value_function}.

We see that the value functions are differentiable even at the optimal threshold levels $A^*$ and this confirms the smooth fit as in Remark \ref{remark_fit} because $X$ is of unbounded variation with $\sigma > 0$.  

In order to verify that these are indeed optimal, we focus on the case $\gamma = 0.05$ and plot in Figure \ref{figure_comparison} the value function $u_{A^*}$ in comparison to the expected values of ``perturbed" strategies $u_A(\cdot)$ for $A=A^*-2, A^*-1,A^*+1,A^*+2$.  Notice that these can be computed by the formula  \eqref{value_function_for_any_A}.  For any choice of $A$, it is easy to see that $u_A$ is continuous as in Remark \ref{remark_fit} but fails to be differentiable at $A \neq A^*$.  We can confirm in all six cases that $u_{A^*}$ indeed dominates $u_A$ for $A \neq A^*$ uniformly in $x \in \R$. This numerically verifies Propositions \ref{theorem_polynomial} and \ref{proposition_general}.

\begin{figure}[htbp]
\begin{center}
\begin{minipage}{1.0\textwidth}
\centering
\begin{tabular}{cc}
\includegraphics[scale=0.6]{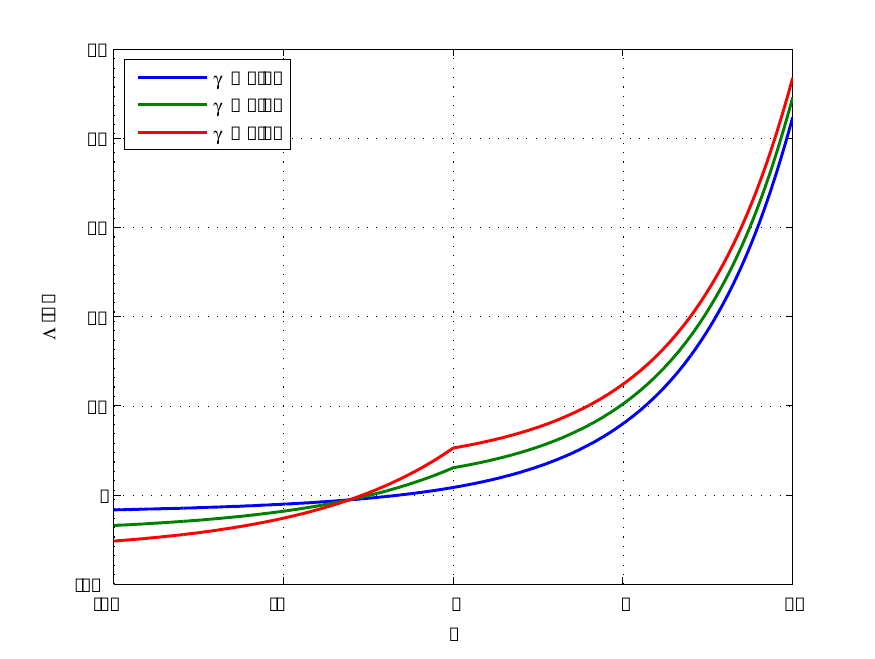}  & \includegraphics[scale=0.6]{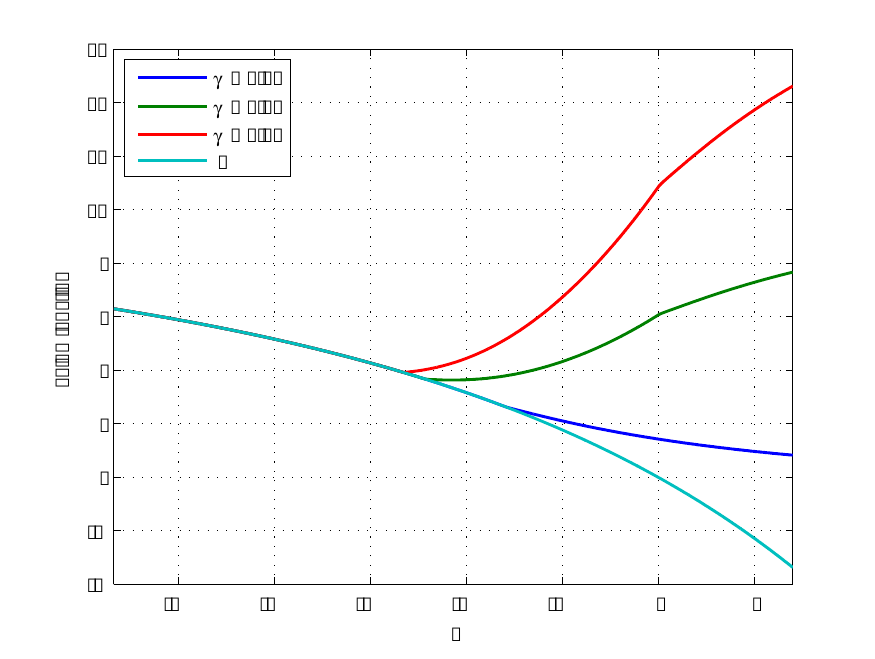} \\
$\Lambda(\cdot)$ for $f^{(sim)}$& value function for $f^{(sim)}$\\
\includegraphics[scale=0.6]{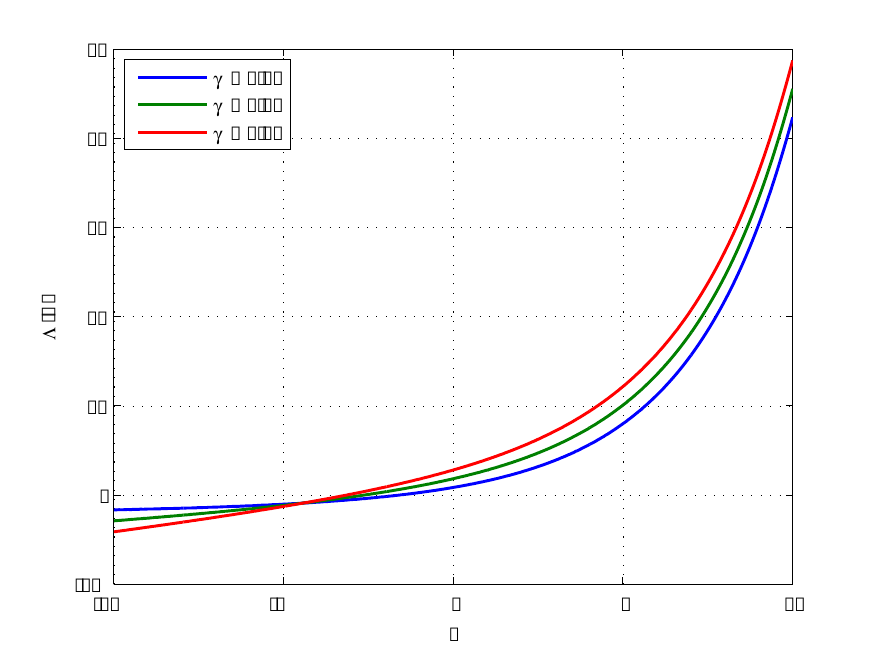}  & \includegraphics[scale=0.6]{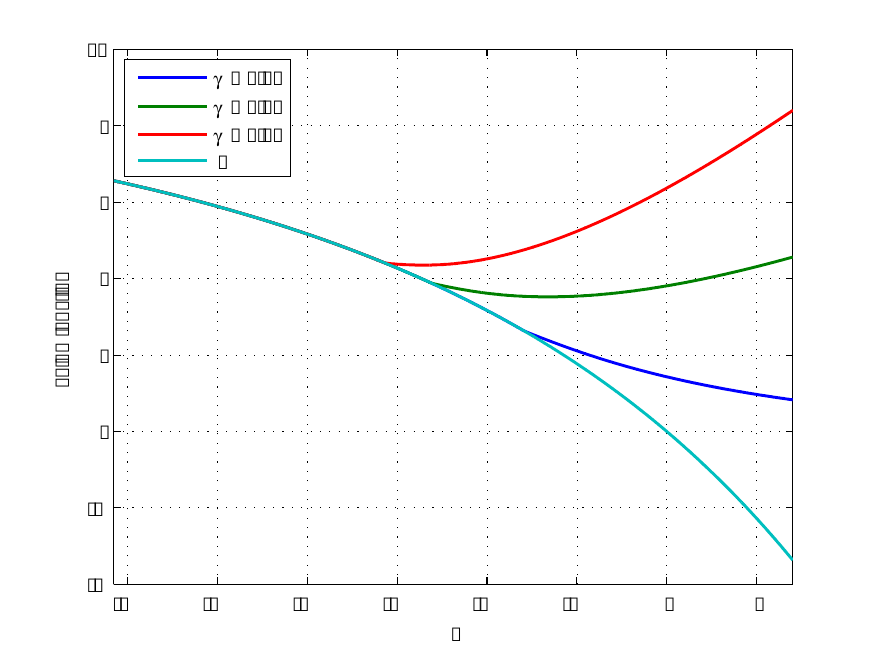} \\
$\Lambda(\cdot)$ for $f^{(lin)}$& value function for $f^{(lin)}$ \\
\includegraphics[scale=0.6]{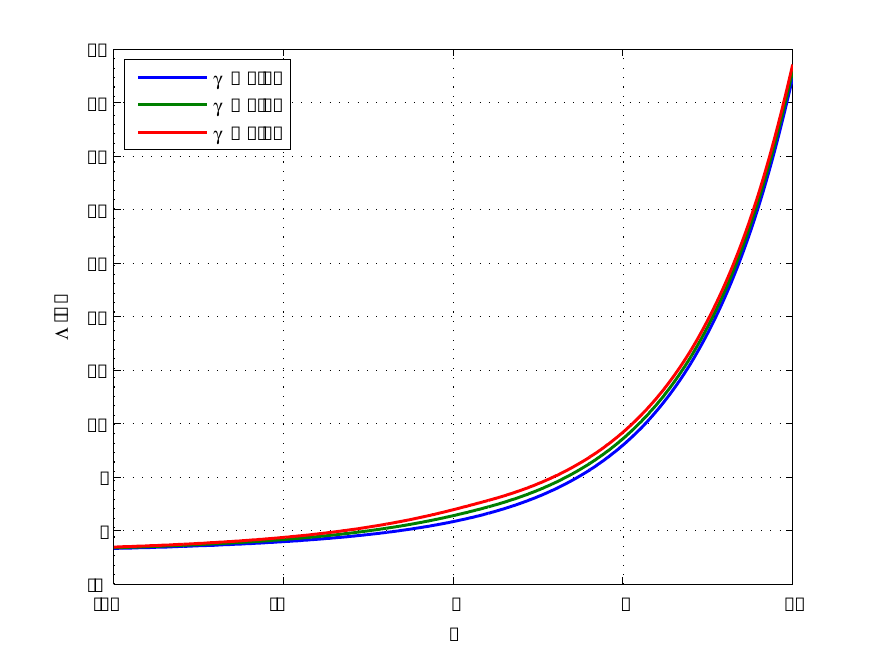}  & \includegraphics[scale=0.6]{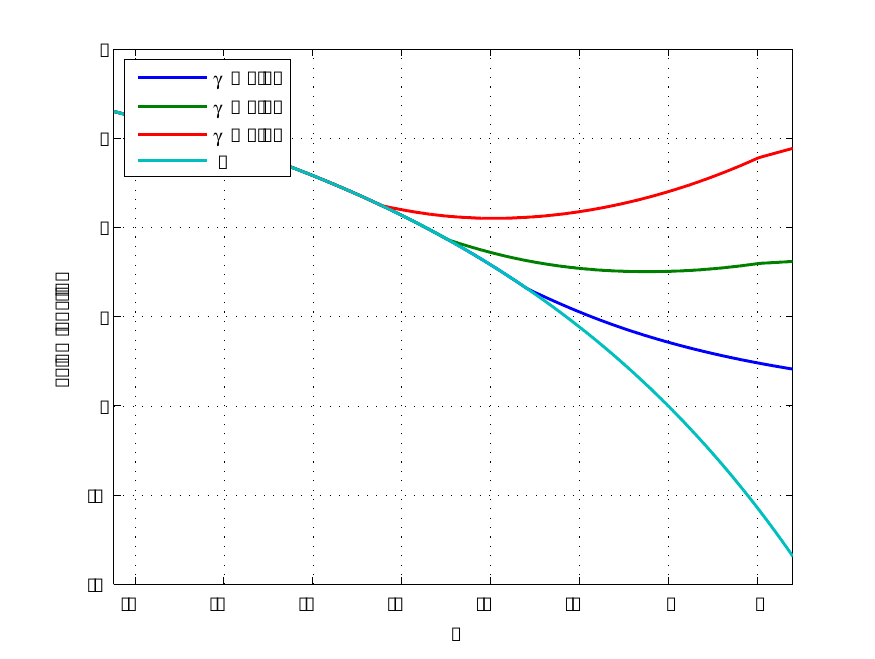} \\
$\Lambda(\cdot)$ for  $f^{(exp)}$ & value function for  $f^{(exp)}$
\end{tabular}
\end{minipage}
\caption{\small{Plots of $\Lambda$ (left) and the value function $u_{A^*}$ along with $g$ (right) when $g = g^{(exp)}$ for the cases (i)-(iii) for $f$.  The function $\Lambda$ is monotonically increasing and its unique zero becomes $A^*$.  The value function is such that it is smoothly pasted at the level $A^*$.}} \label{figure_exp}
\end{center}
\end{figure}

\begin{figure}[htbp]
\begin{center}
\begin{minipage}{1.0\textwidth}
\centering
\begin{tabular}{cc}
\includegraphics[scale=0.6]{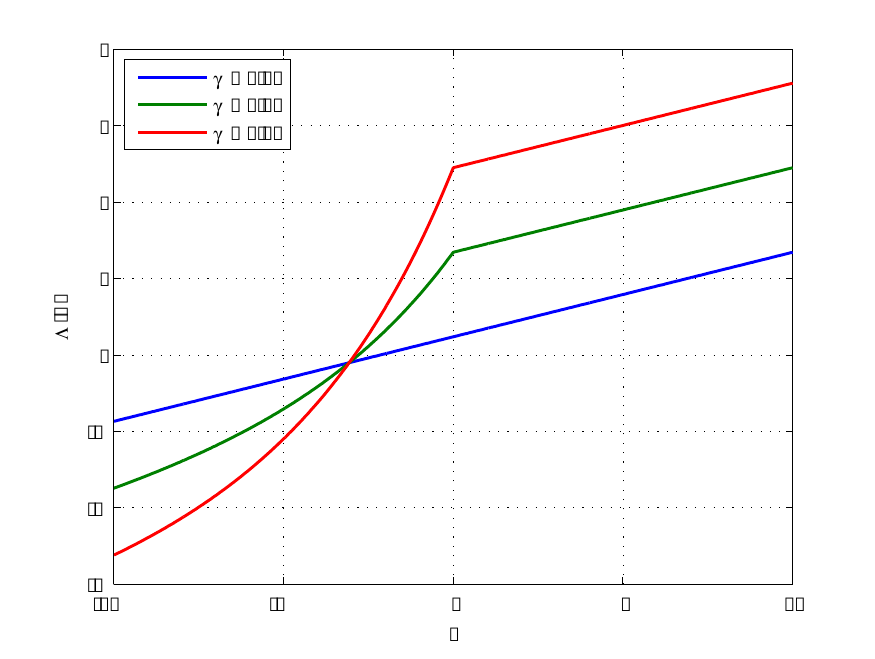}  & \includegraphics[scale=0.6]{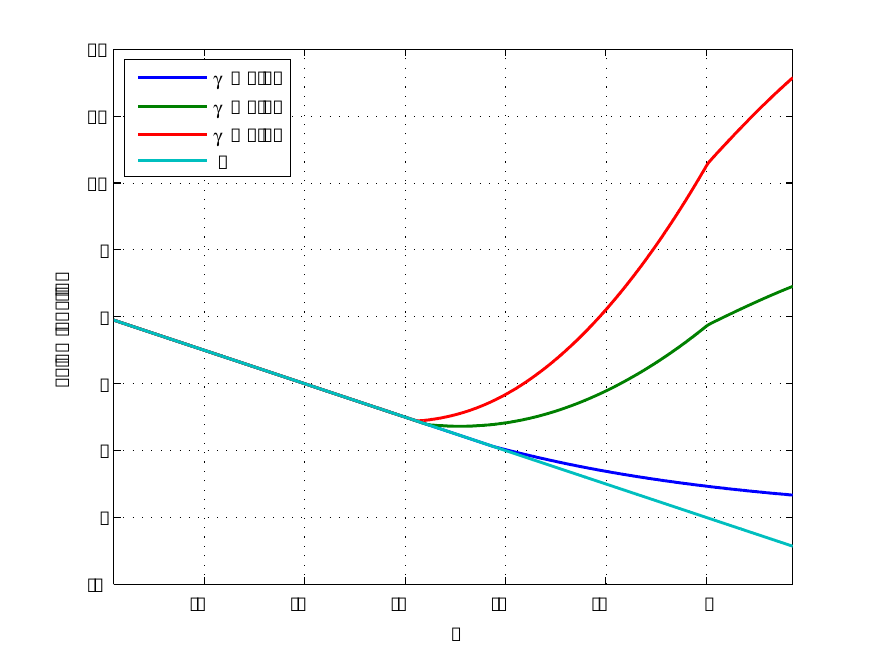} \\
$\Lambda(\cdot)$ for $f^{(sim)}$& value function for $f^{(sim)}$\\
\includegraphics[scale=0.6]{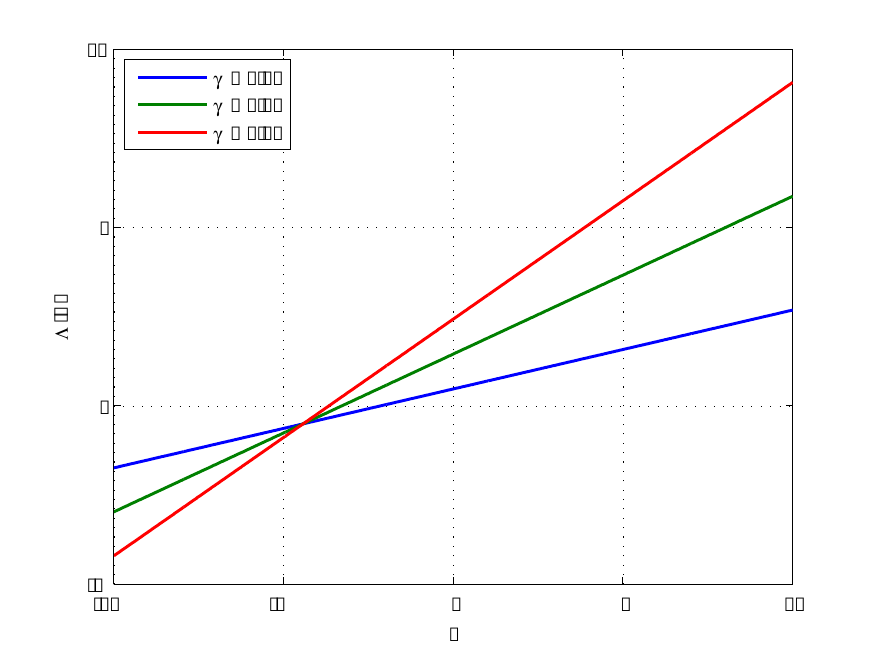}  & \includegraphics[scale=0.6]{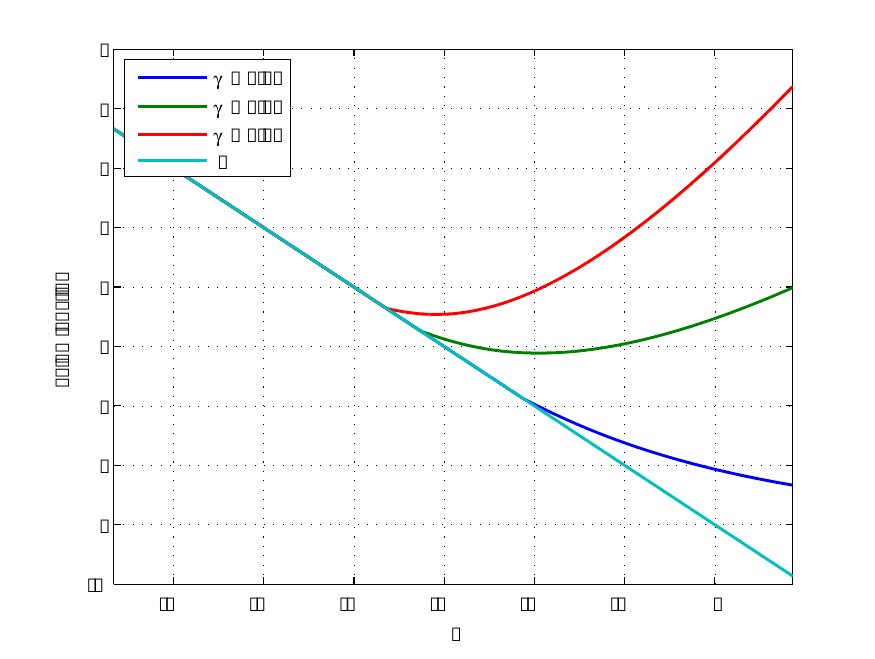} \\
$\Lambda(\cdot)$ for $f^{(lin)}$& value function for $f^{(lin)}$ \\
\includegraphics[scale=0.6]{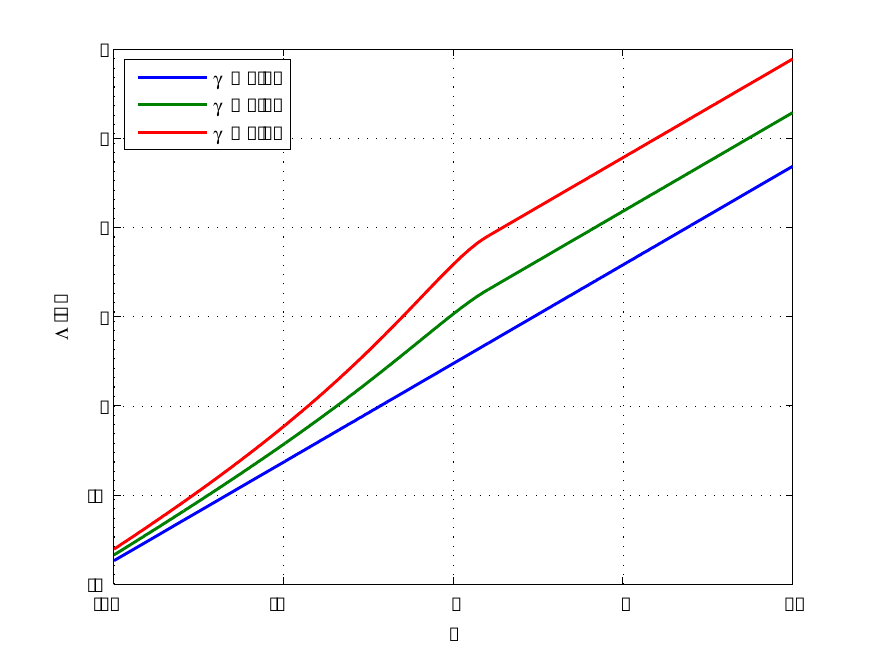}  & \includegraphics[scale=0.6]{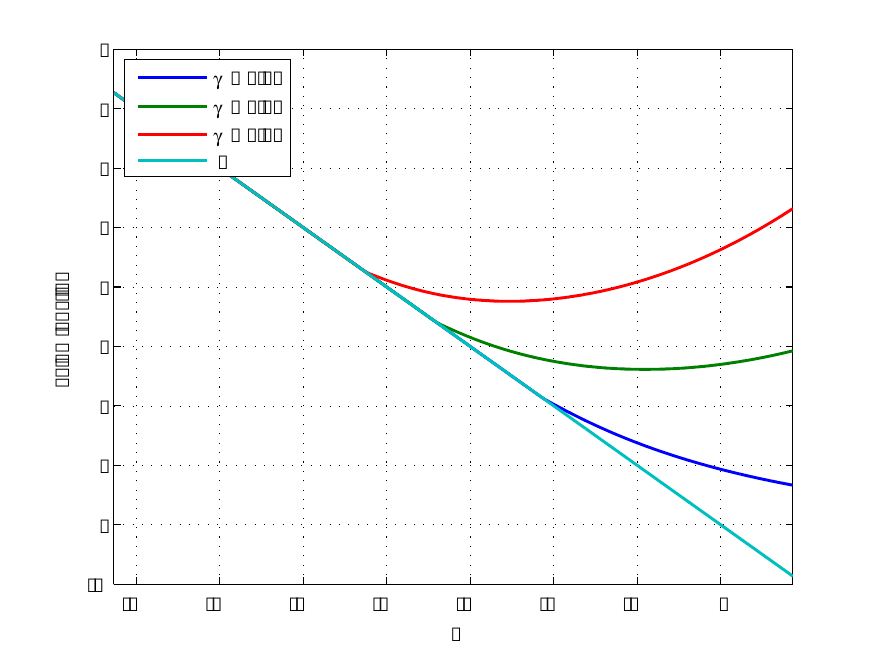} \\
$\Lambda(\cdot)$ for  $f^{(exp)}$ & value function for  $f^{(exp)}$
\end{tabular}
\end{minipage}
\caption{\small{Plots of $\Lambda$ (left) and the value function $u_{A^*}$ along with $g$ (right) when $g = g^{(lin)}$ for the cases (i)-(iii) for $f$. }} \label{figure_lin}
\end{center}
\end{figure}

\begin{figure}[htbp]
\begin{center}
\begin{minipage}{1.0\textwidth}
\centering
\begin{tabular}{cc}
\includegraphics[scale=0.6]{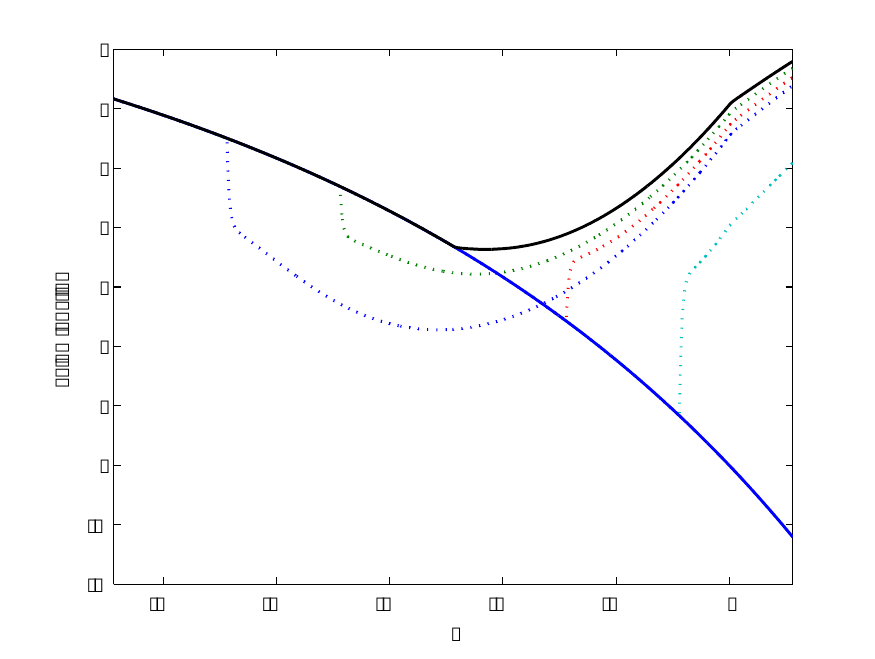}  & \includegraphics[scale=0.6]{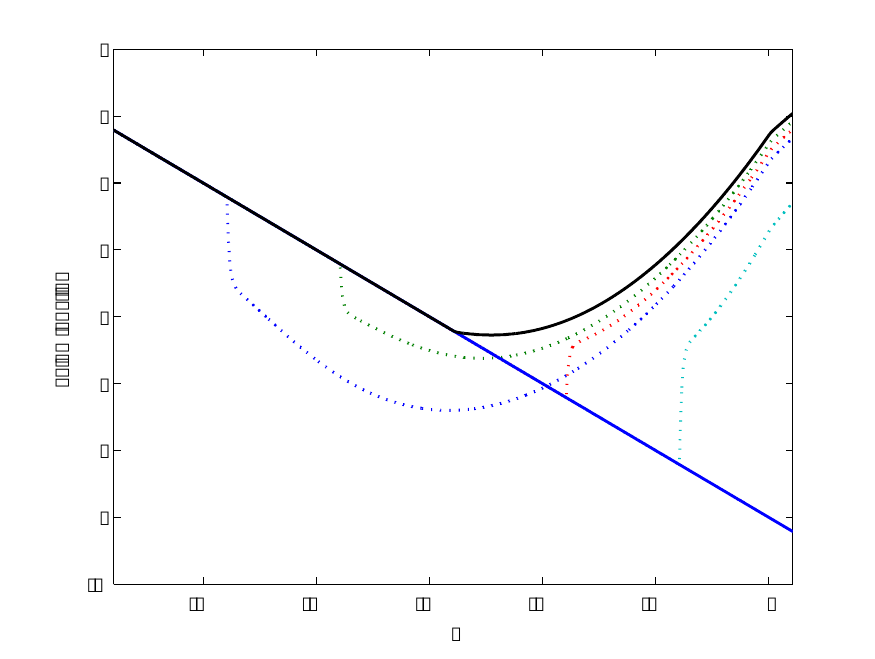} \\
(a-i)  $g^{(exp)}$ with $f^{(sim)}$& (b-i)  $g^{(lin)}$ with $f^{(sim)}$ \\
\includegraphics[scale=0.6]{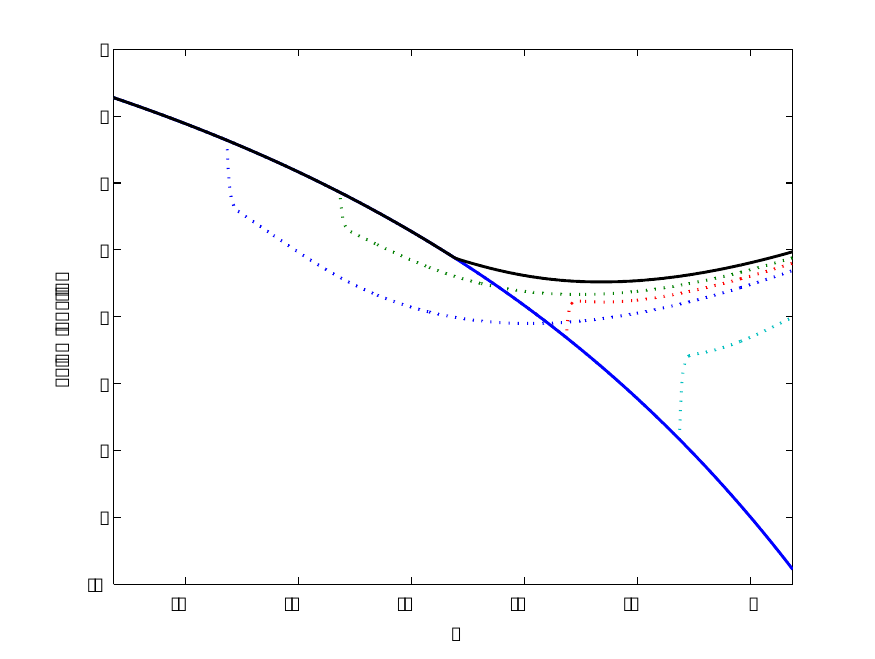}  & \includegraphics[scale=0.6]{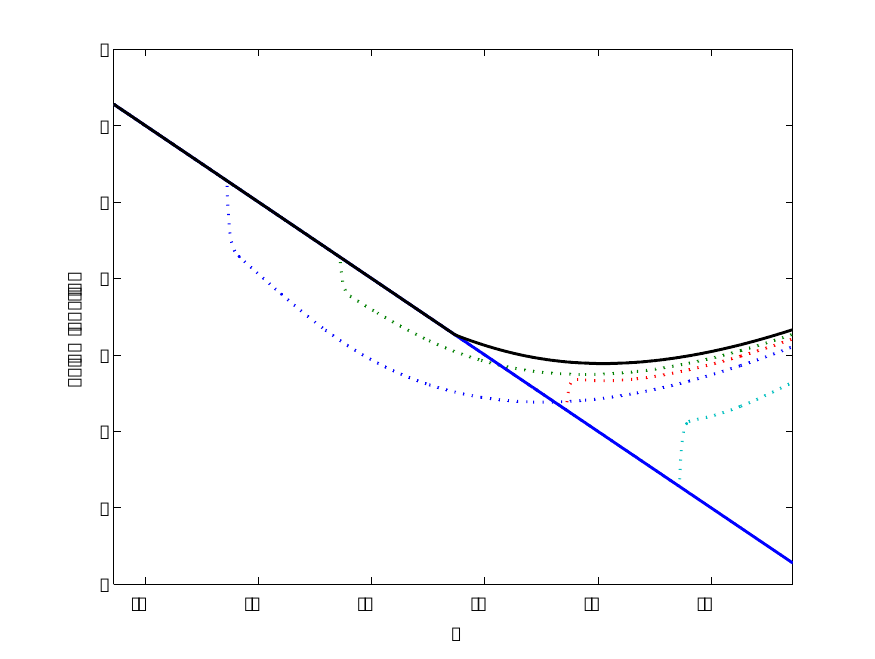} \\(a-ii)  $g^{(exp)}$ with $f^{(lin)}$& (b-ii)  $g^{(lin)}$ with $f^{(lin)}$ \\
\includegraphics[scale=0.6]{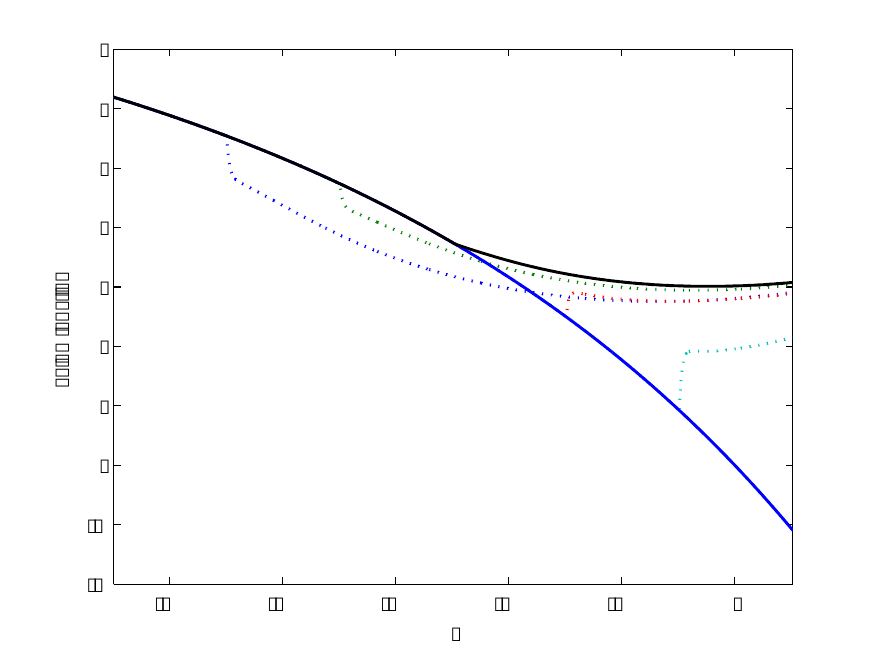}  & \includegraphics[scale=0.6]{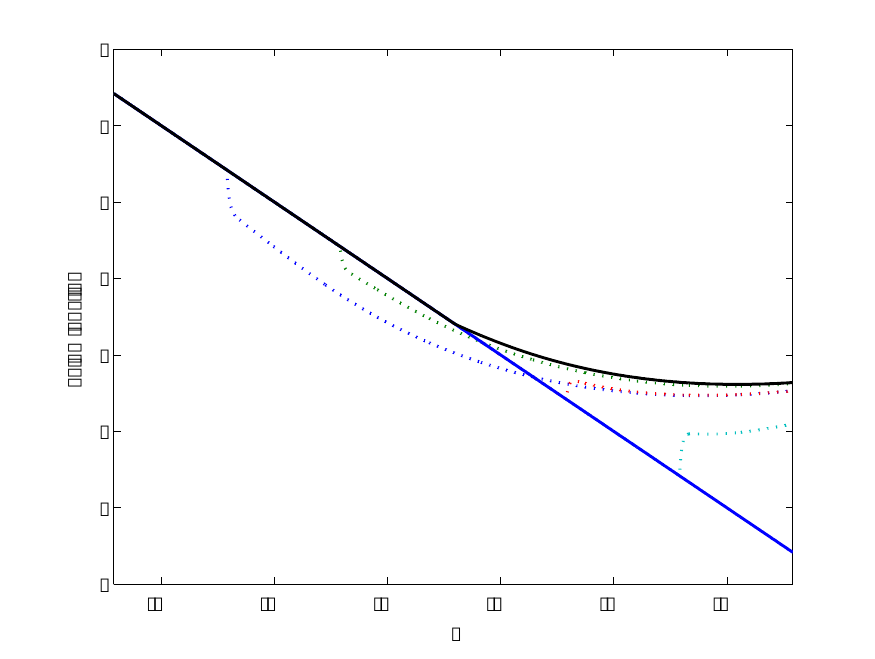} \\
(a-iii)  $g^{(exp)}$ with $f^{(exp)}$& (b-iii)  $g^{(lin)}$ with $f^{(exp)}$ \\
\end{tabular}
\end{minipage}
\caption{\small{Illustration of optimality for the one-stage problem for each combination of $g$ and $f$.  The value function $u_{A^*}$  is plotted in solid black and the stopping value $g$ is in solid blue.  As a comparison, the expected payoff functions corresponding to the perturbed strategies  $u_A(\cdot)$ for $A=A^*-2, A^*-1,A^*+1,A^*+2$ are shown in dotted.}} \label{figure_comparison}
\end{center}
\end{figure}

\subsection{Numerical results on the multiple-stage problem}  We now move onto the multiple-stage problem.  We assume $M=3$ for brevity and use for $f$ and $g$ the functions (a)-(b) and (i)-(iii) defined for the one-stage problem.

For $m=1,3$, we assume $g_m = g^{(exp)}$ for some $a_m = (a_{mi})_{1 \leq i \leq N_m}$ and $c_m= (c_{mi})_{1 \leq i \leq N_m}$  with $N_m = 4$ and a fixed value $K = 10$, whereas for $m=2$, $g_2 = g^{(lin)}$ for some $\alpha_2$.  Also,  we let (i) $f_1 = \gamma_1 f^{(sim)}$ with $I_1 = (-\infty, 0)$, $I_2 = [0,\infty)$, $f^{(1)}=-10$ and  $f^{(2)}=10$, (ii) $f_2 = \gamma_2 f^{(lin)}$ for $b_1=1$ and  $b_2=0$ and  (iii) $f_3 = \gamma_3 f^{(exp)}$ for $L=B=1$.

We conduct a number of experiments for various values of  coefficients.  By using the algorithm given in Subsection \ref{subsection_multi_stage}, the optimal threshold levels $A^* = (A^{*(1)}, A^{*(2)}, A^{*(3)})$ take values among $\{A_1^*, A_2^*, A_3^*, A_{\{1,2\}}^*, A_{\{2,3\}}^*, A_{\{1,2,3\}}^* \}$ and satisfy one of the following four cases:
\begin{enumerate}
\item[] Case 1: $A^{*(1)}=A^{*(2)}= A^{*(3)}$; 
\item[] Case 2: $A^{*(1)}>A^{*(2)}= A^{*(3)}$; 
\item[]  Case 3: $A^{*(1)}=A^{*(2)}> A^{*(3)}$;
\item[]  Case 4: $A^{*(1)} > A^{*(2)}> A^{*(3)}$.
\end{enumerate}

Here we use a random number generator to sample $a_m$, $c_m$ for $m = 1,3$, $\alpha_2$, and $\gamma_m$ for each $1 \leq m \leq 3$  until we attain each of Cases 1 to 4.  The generated parameters and the corresponding threshold levels are summarized in Table \ref{table_A_param}.  In order to validate the optimality of the strategy $(\tau_{A^{*(1)}}, \tau_{A^{*(2)}}, \tau_{A^{*(3)}})$, we compare
in Figure \ref{figure_comparison_mult} the value function with those of perturbed strategies $(\tau_{\widetilde{A}_k^{(1)}}, \tau_{\widetilde{A}_k^{(2)}}, \tau_{\widetilde{A}_k^{(3)}})$, $1 \leq k \leq 6$, where
\begin{align*}\widetilde{A}_k := (A^{*(1)}, A^{*(2)},  A^{*(3)}) + \delta_k
\end{align*}
 with $\delta_1 := (1,0,0)$,  $\delta_2 := (1,1,0)$,  $\delta_3 := (1,1,1)$,  $\delta_4 := (0,0,-1)$,  $\delta_5 := (0,-1,-1)$ and $\delta_6 := (-1,-1,-1)$.
It is clear that $(\tau_{\widetilde{A}_k^{(1)}}, \tau_{\widetilde{A}_k^{(2)}}, \tau_{\widetilde{A}_k^{(3)}}) \in \ST_3 \subset \S_3$ because $\widetilde{A}^{(1)}_k \geq \widetilde{A}^{(2)}_k \geq \widetilde{A}^{(3)}_k$ by construction. Figure \ref{figure_comparison_mult} suggests in all cases that the value obtained by $(\tau_{A^{*(1)}}, \tau_{A^{*(2)}}, \tau_{A^{*(3)}})$  dominates uniformly over $x$ those obtained by the perturbed strategies.  These results are indeed consistent with our main theoretical results.
In view of Figure \ref{figure_comparison_mult}, we also observe that there are up to three kinks (at $A^*$) in the value function although these are still differentiable.  This is again due to smooth fit as in Remark \ref{remark_fit}.  The perturbed strategies on the other hand fail to be differentiable while they are still continuous.

\begin{table}[htb]
 \begin{tabular}{|c||c|c|} \hline
 & $a_1$ &  $a_3$ \\ \hline 
Case 1 &(0.49,   0.19,   0.17,   0.03)   & (0.05,   0.24,   0.46,   0.13)\\
Case 2  &   (0.47,   0.17,   0.06,   0.12)  &
   (0.24,   0.18,   0.19,   0.05)   \\ 
Case 3  &    (0.04,   0.06,   0.01,   0.12) &
   (0.01,   0.11,   0.26,   0.05)  \\
Case 4  &   (0.39,   0.28,   0.17,   0.16) &
   (0.06,   0.01,  0.40,   0.08)    \\
\hline
  \end{tabular} 
\\  \vspace{0.3cm}
 \begin{tabular}{|c||c|c|} \hline
 & $c_1$ &    $c_3$ \\ \hline 
Case 1 & (2.11,   2.09,   3.51,   3.49) &
   (4.71,   1.51,   2.70,   0.89) \\
Case 2  &     (0.66,   2.88,   1.77   0.22) &
   (4.78,   1.17,   0.08,   3.24)  \\ 
Case 3  &       (1.84,   2.39,   4.20,   2.23) &
   (3.01,   3.72,   2.57,   4.20) \\
Case 4  &     (3.01,   3.45,   0.42,   0.76) &
   (3.27,   2.25,   4.57,   2.69)   \\
\hline
  \end{tabular} 
 \\  \vspace{0.3cm}
  \begin{tabular}{|c||c|} \hline
 & $\alpha_2$  \\ \hline 
Case 1  &0.9991 \\
Case 2   &  0.6477   \\ 
Case 3  &0.6265\\
Case 4  &  0.0782    \\
\hline
  \end{tabular}   \hspace{0.5cm}
  \begin{tabular}{|c||c|c|c|} \hline
 & $\gamma_1$ (simple)&  $\gamma_2$ (linear) &  $\gamma_3$ (exponential)  \\ \hline 
Case 1  &0.2920 &   0.4317 &   0.0155\\
Case 2   &  0.4173 &   0.0497 &   0.9027  \\ 
Case 3  &0.3070 &   0.0611 &   0.2195  \\
Case 4  &  0.0759  & 0.0540 &  0.5308  \\
\hline
  \end{tabular} 
 \\  \vspace{0.3cm}
  \begin{tabular}{|c||c|c|c|c|c|c||c|c|c|} \hline
 & $A_1^*$ &  $A_2^*$ &  $A_3^*$ & $A_{\{1,2\}}^*$ &  $A_{\{2,3\}}^*$ &  $A_{\{1,2,3\}}^*$  & $A^{*(1)}$ &  $A^{*(2)}$ &  $A^{*(3)}$ \\ \hline 
Case 1 &-2.44 &  -2.83 &  -1.39 &  -2.59&-2.03&  -2.21 &  -2.21 &  -2.21&  -2.21\\
Case 2  & -0.48&  -2.31&  -2.18&  -0.85&  -2.22&  -1.34&  -0.48&  -2.22&  -2.22\\ 
Case 3  & -3.15&  -2.35&  -5.67&  -2.85&  -4.14&-3.75& -2.85&  -2.85&  -5.67 \\
Case 4  &-0.76&  -3.07&  -3.64&  -0.85& -3.59&  -1.89& -0.76&  -3.07&  -3.64 \\
\hline
  \end{tabular} 
\caption{Parameters and threshold levels for the plots in Figure \ref{figure_comparison_mult}.  These values are chosen so that the parameter set satisfies each of Cases 1-4. } \label{table_A_param} 
\end{table} 

\begin{figure}[htbp]
\begin{center}
\begin{minipage}{1.0\textwidth}
\centering
\begin{tabular}{cc}
\includegraphics[scale=0.6]{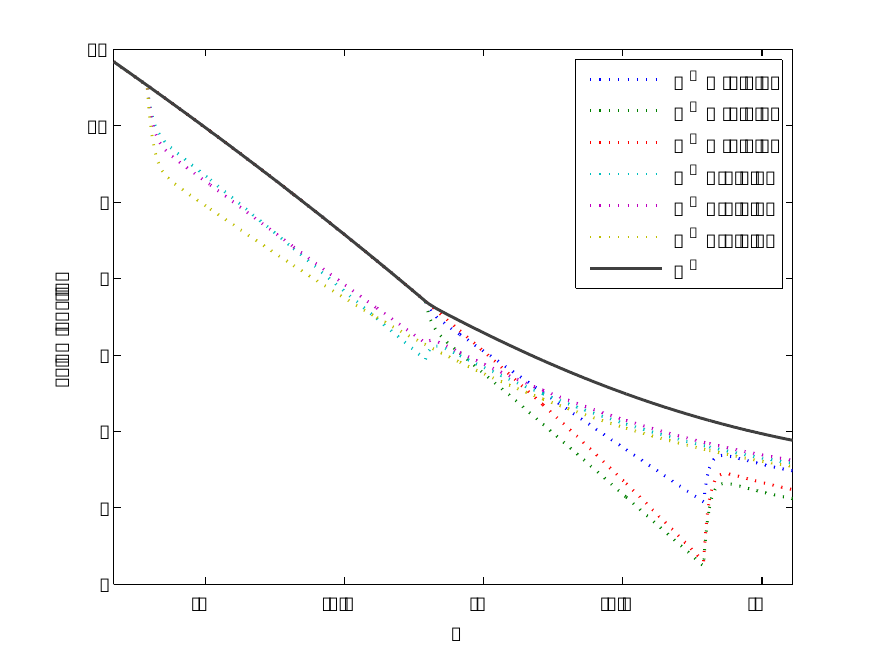}  & \includegraphics[scale=0.6]{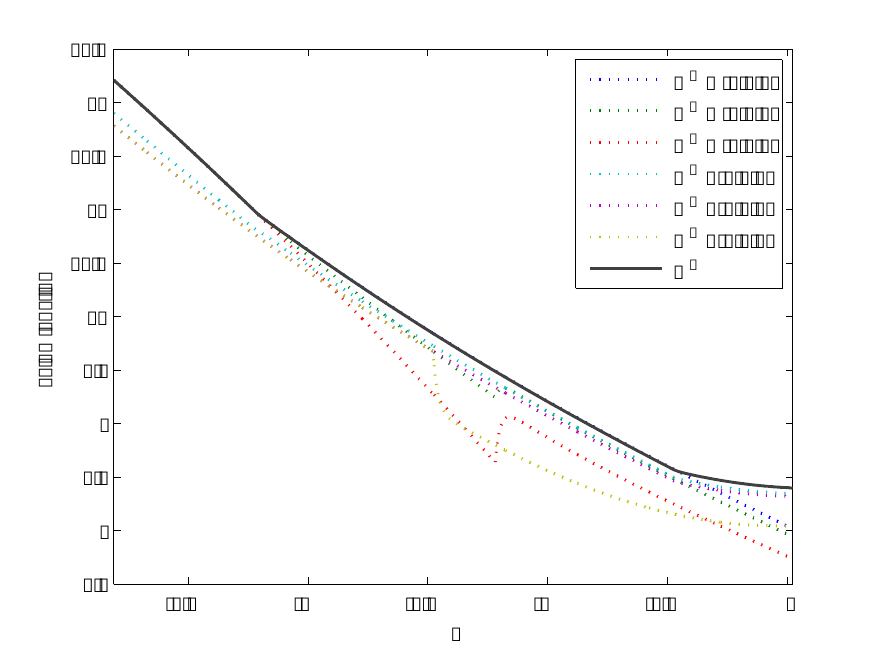} \\
Case 1 & Case 2 \\
\includegraphics[scale=0.6]{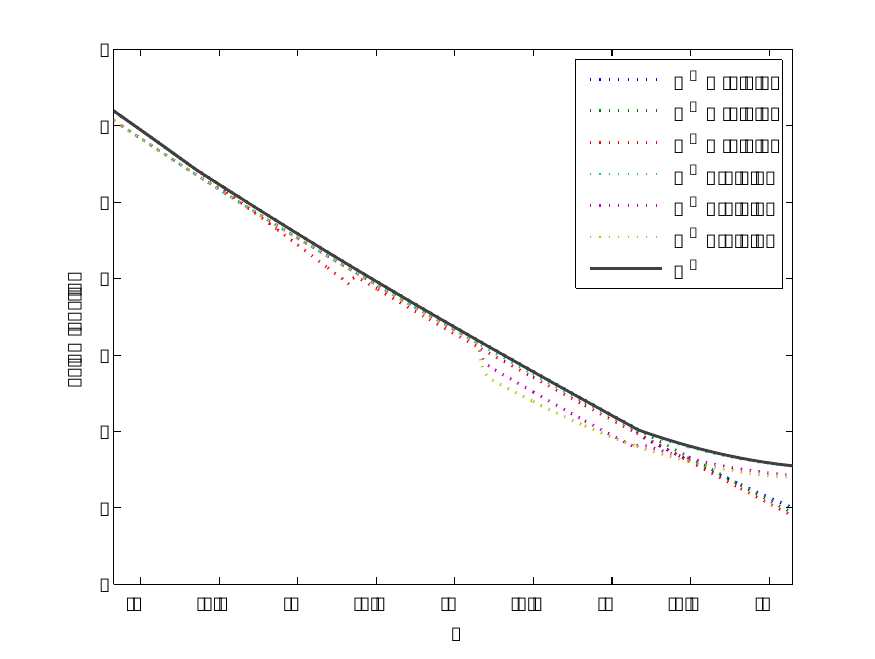}  & \includegraphics[scale=0.6]{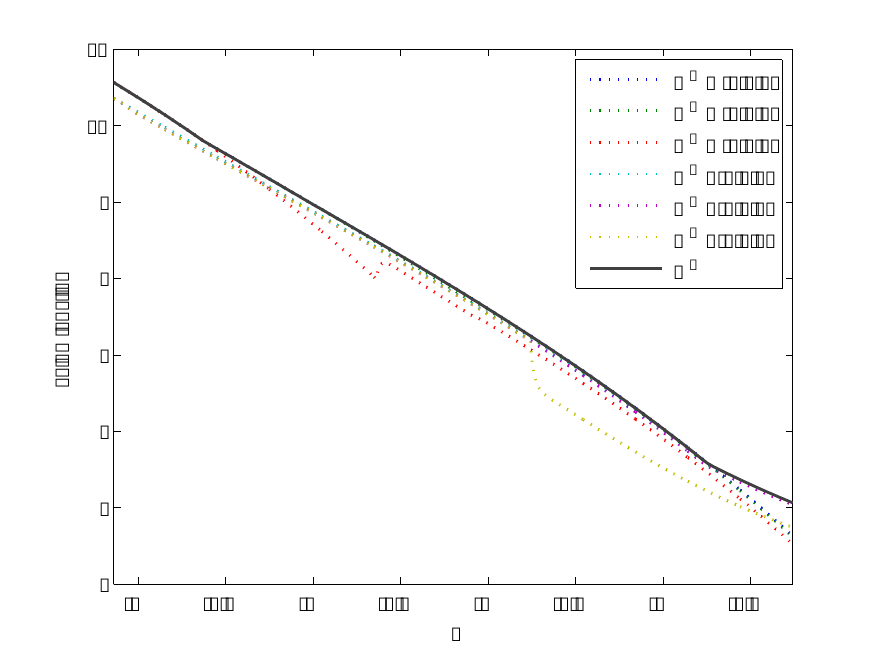} \\ Case 3 & Case 4
\end{tabular}
\end{minipage}
\caption{\small{Illustration of optimality for the multiple-stage problem.  The parameters for $g$ and $f$ are shown in Table \ref{table_A_param}.  The value function $U^{(3)}$  is plotted in solid black against the expected payoff functions corresponding to the perturbed strategies with threshold levels $\widetilde{A}$.}} \label{figure_comparison_mult}
\end{center}
\end{figure}

\section{Concluding Remarks}  \label{section_conclusion}In this paper, we studied a wide class of optimal stopping problems for a general spectrally negative \lev process and extended them to multiple-stopping.  Our framework is applicable to a wide range of settings particularly in real option problems where the firm withdraws from a project in stages.  Our analytical results suggest that the optimal solutions can be characterized by the threshold levels that are zeros of  certain monotone functions, and the corresponding value functions can be expressed in terms of the scale function.  Our  numerical experiments suggest, for the phase-type jump case, that these can be solved instantaneously with high precision.   These tools we developed in this paper are highly valuable and 
can be used flexibly for analysis in real options and other fields of finance and industrial applications.  

%

There are several directions for future research.  First, our results can be pursued for a general \lev process with both positive and negative jumps.  While it makes the problem less tractable, it is expected that these can be done at least for the cases with rational forms of Wiener-Hopf factors such as meromorphic \lev processes \cite{Kuznetsov_2010} and phase-type \lev processes \cite{Asmussen_2004}. Second, by using phase-type fitting, one can approximate any \lev process by those with phase-type jumps as in Section \ref{section_numer}. By calibrating with real financial and industrial data as in \cite{AsmussenMadanPistorius07},  one can conduct detailed empirical analyses on optimal stopping strategies and the value functions.  Finally, it is an interesting extension to consider ``swing option type" multiple-stopping with refraction periods as in \cite{Carmona_dayanik, Touzi_Carmona}  where  any two consecutive stoppings must be separated by  fixed constants.


\appendix


\section{Proofs}

\subsection{Proof of Lemma \ref{lemma_lambda_A_polynomial}} \label{proof_lemma_lambda_A_polynomial}
Let $g_1(x) := -bx$, $x \in \R$. We have
\begin{align*}\rho_{g_1,A}^{(r)} &= -b \int_{(0,\infty)}\Pi (\diff u)   \int_0^{u} e^{-\Phi_r z} (z-u) \diff z  = -\frac b {\Phi_r^2}\int_{(0,\infty)} \Pi (\diff u)   (1- e^{-\Phi_r u} - \Phi_r u),
\end{align*}
and hence
\begin{align*}
\Lambda(A; 0, g_1) &:= -\frac r {\lapinv} g_1(A)  - \frac {\sigma^2} 2 g_1'(A) + \rho_{g_1,A}^{(r)}=  b \Big[ \frac r {\lapinv} A  + \frac {\sigma^2} 2  -  \frac 1 {\Phi_r^2} \int_{(0,\infty)}\Pi (\diff u)   (1- e^{-\Phi_r u} -  \Phi_r u) \Big] \\
&=   \frac b {\Phi_r^2}\Big[ c  \Phi_r
+\frac{\sigma^2} 2  \Phi_r^2 +\int_{( 0,\infty)}(e^{- \Phi_r
u}-1+ \Phi_r u 1_{\{0 <u<1\}})\,\Pi(\diff u) \Big] \\
&- \frac b {\Phi_r} \Big[ c   - \int_{[1,\infty)} u \Pi(\diff u)  - r A \Big] \\
&= b \Big[ \frac r  {\Phi_r^2}- \frac {\psi'(0+) -r A } {\Phi_r} \Big].
\end{align*}
Let $g_2(x) = -\sum_{i=1}^N c_i e^{a_i x}$, $x \in \R$.  Then,
\begin{align}
\rho_{g_2,A}^{(r)} = - \sum_{i=1}^N c_i e^{a_i A} \int_{(0,\infty)} \Pi (\diff u) \int_0^u e^{-\Phi_r z} (e^{a_i(z-u)} - 1) \diff z. \label{rho_this_case}
\end{align}
(Case 1) First suppose $a_i \neq \Phi_r$ for all $1 \leq i \leq N$.   Simple algebra gives
\begin{align}
\Lambda(A)= -\frac r {\lapinv}K + b \Big( \frac r  {\Phi_r^2}+ \frac {r A - \psi'(0+)} {\Phi_r} \Big) + \sum_{i=1}^N c_i e^{a_i A} \frac {M_r^{(a_i)}} {\lapinv}  + \Psi_f(A) \label{Psi_american}
\end{align}
where
\begin{align*}
M_r^{(a)} &:= r + \frac {a \sigma^2} 2 \lapinv + \int_{(0,\infty)} \Pi (\diff u) \Big[  (1 - e^{-\lapinv u}) - e^{-a u} (1 - e^{-(\lapinv-a) u}) \frac {\lapinv} {\lapinv -a}\Big], \quad a \in \R \backslash \{\Phi_r \}.
\end{align*}
By the definition of $\psi$ and $\Phi_r$, we rewrite $M_r^{(a)}$ as
\begin{align*}
\begin{split}
 & r + \frac {a \sigma^2} 2  \lapinv + \int_{(0,\infty)} \Pi (\diff u) \left[  (1 - e^{-\lapinv u} - \lapinv u 1_{\{ 0 < u < 1 \}})  - e^{-au} (1 - e^{-(\lapinv-a) u}) \frac {\lapinv} {\lapinv -a} + \lapinv u 1_{\{ 0 < u < 1\}}\right] \\
&= \Big( c + a \sigma^2 - \int_{(0,1)} u (e^{-a u}-1) \Pi(\diff u) \Big) \lapinv + \frac {\sigma^2} 2 \lapinv (\lapinv - a ) \\ &\qquad - \frac {\lapinv} {\lapinv -a} \int_{(0,\infty)} \Pi (\diff u) e^{-a u} \left( 1 - e^{-(\lapinv-a) u} - (\lapinv-a) u 1_{\{0 <  u <1 \}} \right) \\
&= \frac {\lapinv} {\lapinv - a} \psi_a(\lapinv-a), 
\end{split} 
\end{align*}
where the last equality holds by \eqref{psi_a}. 
On the other hand, $\psi_a(\lapinv-a) = \psi(\lapinv) - \psi(a) = r - \psi(a)$; see page 213 of \cite{Kyprianou_2006}.
Hence
\begin{align*}
\lapinv \varpi_r(a) = \frac {\lapinv} {\lapinv - a} (r- \psi(a))  &= \frac {\lapinv} {\lapinv - a} \psi_a(\lapinv-a),
\end{align*}
which shows for the case $a_i \neq \Phi_r$ for all $1 \leq i \leq N$.   

(Case 2) Suppose $a_j = \Phi_r$ for some $1 \leq j \leq N$ (with $a_i \neq a_j$ for $i \neq j$ by assumption).  Take a sequence of (strictly) increasing sequence $a_j^{(m)} \uparrow a_j = \Phi_r$.
Then a modification of \eqref{Large_lambda} with $a_j$ replaced with $a_j^{(m)}$ is by Case 1
\begin{align*}
\Lambda^{(m)}(A)= -\frac r {\lapinv}K + b \Big( \frac r  {\Phi_r^2}+ \frac {r A - \psi'(0+)} {\Phi_r} \Big) + \sum_{1 \leq i \leq N, i \neq j} c_i e^{a_i A}  \varpi_r(a_i)  + c_j e^{a_j^{(m)} A}  \varpi_r(a_j^{(m)})  + \Psi_f(A). 
\end{align*}
By the definition of $\varpi_r$ as in \eqref{varpi}, we have
\begin{align*}
\lim_{m \uparrow \infty}\Lambda^{(m)}(A)= -\frac r {\lapinv}K +  b \Big( \frac r  {\Phi_r^2}+ \frac {r A - \psi'(0+)} {\Phi_r} \Big)  + \sum_{i=1}^N c_i e^{a_i A} \varpi_r(a_i)  + \Psi_f(A).
\end{align*}
On the other hand, in view of \eqref{rho_this_case}, its integrand is monotone in $a$.  Hence by the monotone convergence theorem and because $g$ and $g'$ are continuous in $a$, $\lim_{m \uparrow \infty}\Lambda^{(m)}(A) = \Lambda(A)$, and the proof is complete for Case 2.

\subsection{Proof of Lemma \ref{lemma_smoothness}} \label{proof_lemma_smoothness}Because $g(x)$ is infinitely differentiable, the results are clear for $x \in (-\infty, A^*)$.  Hence we show for $x \in (A^*, \infty)$.
Because $W^{(r)}(y)$ is differentiable on $y > 0$ as in Remark \ref{remark_smoothness_zero}(1), $K Z^{(r)} (x-A^*) - b \Big[ \overline{Z}^{(r)}(x-A^*) + \big(A^*- \frac {\psi'(0+)} r \big)Z^{(r)}(x-A^*) + \frac {\psi'(0+)} r \Big] - \sum_{i=1}^N c_i e^{a_i x} Z_{a_i}^{(r - \psi(a_i))} (x-A^*)$ is twice differentiable. 

Regarding $\Theta_f(x;A^*)$,  integration by parts thanks to the continuity of $f$ gives (with $\overline{W}^{(r)}(x) := \int_0^x W^{(r)}(y) \diff y$, $x \in \R$)
\begin{align*}
\Theta_f(x;A^*) = f(A^*) \overline{W}^{(r)} (x-A^*) + \int_{A^*}^x  f'(y) \overline{W}^{(r)}(x-y) \diff y.
\end{align*}
It is differentiable with
\begin{align*}
\Theta_f'(x;A^*) = f(A^*)  {W^{(r)}(x-A^*)} + \int_{A^*}^x f'(y) {W^{(r)}(x-y)} \diff y.
\end{align*}
When $X$ is of unbounded variation, because $W^{(r)}(0) = 0$ as in Remark \ref{remark_smoothness_zero}(2), $\Theta_f(x;A^*)$ is twice-differentiable with
\begin{align*}
\Theta_f''(x;A^*) = f(A^*)  {W^{(r)'}(x-A^*)}+ \int_{A^*}^x f'(y) {W^{(r)'}(x-y)} \diff y.
\end{align*}


\subsection{Proof of Proposition  \ref{theorem_polynomial}} \label{proof_theorem_polynomial}


(i) Suppose $-\infty < A^* \leq \infty$.  By directly using the results of \cite{Egami-Yamazaki-2011} (Lemma 3.7 and Proposition 3.4), we obtain
\begin{align} \label{results_egami_yamazaki}
\begin{split}
&(\mathcal{L}-r) u_{A^*}(x) + f(x) = 0, \quad x \in (A^*, \infty), \\
&u_{A^*}(x) \geq g(x), \quad x \in \R,
\end{split}
\end{align}
where $\mathcal{L}$ is the infinitesimal generator of $X$ applied to a sufficiently smooth function $h$, i.e.,
\begin{align*}
\mathcal{L} h(x) = c h'(x) + \frac 1 2 \sigma^2 h''(x) + \int_{(0,\infty)} [h(x-z) - h(x) + h'(x) z 1_{\{0 < z < 1\}}] \Pi (\diff z).
\end{align*}

\begin{lemma} \label{lemma_superharmonic} If  $-\infty < A^* \leq \infty$, we have $(\mathcal{L}-r) u_{A^*}(x) + f(x) \leq 0$ on $x \in (-\infty, A^*)$.
\end{lemma}
\begin{proof}Fix $-\infty < A^* < \infty$.  First, if we define $g_l(x) := x$, $x \in \R$,
\begin{align*}
(\mathcal{L}-r) g_l(x) = \psi'(0+)- rx.
\end{align*}
By the definition of $\psi$, if we define  $g_e (x) := e^{ax}$, $x \in \R$,
\begin{align*}
\mathcal{L} g_e(x) = e^{ax} \Big[ ca + \frac 1 2 \sigma^2 a^2+ \int_{(0,\infty)} (e^{-az} - 1 + a  z 1_{\{0 < z < 1\}} ) \Pi(\diff z) \Big] = e^{ax} \psi(a)
\end{align*}
 for any $a > 0$ and hence
we have
\begin{align}
(\mathcal{L}-r) g(x) + f(x) = - rK -b(\psi'(0+) -rx)  + \sum_{i=1}^N c_i e^{a_i x} (r - \psi(a_i))  + f(x). \label{american_verification_1}
\end{align}
By how $A^*$ is chosen,
\begin{align}
0 = -r K + b \Big(\frac r {\Phi_r}- {(\psi'(0+) -r A^*)}\Big) + \sum_{i=1}^N c_i e^{a_i A^*}  {\lapinv} \varpi_r(a_i)    + \lapinv \Psi_f(A^*). \label{american_verification_2}
\end{align}

Because $f$ is increasing and $x < A^*$
\begin{align}
\lapinv \Psi_f(A^*) \geq \lapinv \int_0^\infty e^{-\lapinv y} f(x) \diff y = f(x). \label{american_verification_3}
\end{align}
By $A^* \geq x$ and $\Phi_r > 0$,
\begin{align*}
b \Big(\frac r {\Phi_r}- {(\psi'(0+) -r A^*)}\Big) \geq  -b(\psi'(0+) -rx).
\end{align*}
It is also easy to see that
\begin{align}
e^{a_i A^*}{\lapinv} \varpi_r(a_i)  \geq e^{a_i x} (r - \psi(a_i)), \quad 1 \leq i \leq N.  \label{american_verification_4}
\end{align}
Indeed, for the case $r - \psi(a_i) > 0$, we must have $\lapinv-a_i > 0$ and hence \eqref{american_verification_4} holds by $A^* > x$; for the case $r - \psi(a_i) < 0$, the left-hand side is positive while the right-hand side is negative in \eqref{american_verification_4}; for the case $r - \psi(a_i) = 0$, the left-hand side is positive because $\psi'(\Phi_r)$ is, while the right-hand side is zero.
Hence, by \eqref{american_verification_2}-\eqref{american_verification_4}, $(\mathcal{L}-r) u_{A^*}(x) + f(x) \leq 0$ holds. 

This result also holds for the case $A^* = \infty$. In this case,  $0 > -r K + \sum_{i=1}^N c_i e^{a_i \widehat{A}} \lapinv \varpi_r (a_i) + b \big(\frac r {\Phi_r}- {(\psi'(0+) -r \widehat{A})}\big)  + \lapinv \Psi_f(\widehat{A})$,
for any $\widehat{A}\in \R$.  Therefore $(\mathcal{L}-r) g(x) + f(x) < 0$ holds by the same reasoning as in (1) by simply replacing $A^*$ with $\widehat{A}$ for any $\widehat{A} > x$.

\end{proof}

We are now ready to verify the optimality of $u_{A^*}(x)$ for the case $A^* \in (-\infty, \infty]$.   Thanks to Lemma \ref{lemma_smoothness} and the continuous/smooth fit condition as in  Remark \ref{remark_fit}, a version of Meyer-Ito's formula as in Theorem IV.71 of \cite{ProtterBook} (see also Theorem 2.1 of  \cite{sulem}) implies
\begin{align}\label{Decomp}
 e^{-r t}u_{A^*}(X_t)-u_{A^*}(X_0) 
 &= \int_0^t e^{-rs}(\mathcal{L}-r)u_{A^*}(X_{s-})\diff s+ M_t,\nonumber
 \end{align}
with the local martingale part
 \begin{align*}
 M_t &:= \int_0^t \sigma e^{-rs} u_{A^*}'(X_{s-}) \diff B_s -\int_{(0,t]} \int_{(0,1)} e^{-rs}u_{A^*}'(X_{s-})y (N(\diff s\times \diff y)-\Pi(\diff y) \diff s)\\
 &+\int_{(0,t]} \int_{(0, \infty)}e^{-rs}(u_{A^*}(X_{s-}-y)-u_{A^*}(X_{s-})+u_{A^*}'(X_{s-})y1_{\{0 < y < 1\}})(N(\diff s\times \diff y)-\Pi(\diff y) \diff s),
 \end{align*}
where $N(\diff s \times \diff x)$ is the Poisson random measure associated with the dual process $-X$.

Fix any stopping time $\tau$, and define for each $m\in \N$ the stopping time $T_m$ as
$$
T_m :=\inf\{t>0\; : \;|X_t|> m\},
$$
and the martingale process $M=\{M_{t\wedge \tau \wedge T_m}: t\geq 0\}$, with $M_0=0$. 
By optional sampling and because $(\mathcal{L} -r) u_{A^*} + f \leq 0$ via \eqref{results_egami_yamazaki} and Lemma \ref{lemma_superharmonic},
\begin{align*}
\E^x \Big[ e^{-r (t \wedge \tau \wedge T_m)} u_{A^*}(X_{t \wedge \tau \wedge T_m}) + \int_0^{t \wedge \tau \wedge T_m} e^{-rs} f(X_s) \diff s \Big] \leq u_{A^*}(x).
\end{align*}
When $x < A^*$,  $u_{A^*}(x) = g(x) \geq g(A^*) > -\infty$ (the same result holds for $A^* = \infty$ by Remark \ref{remark_A_tilde}).  On the other hand, if $A^* \in (-\infty, \infty)$ and $x  > A^*$, because $\partial u_A(x) / \partial A > 0$ on $(-\infty, A^*)$, the value  $u_{A^*}(x)$ is bounded from below by a limit:
\begin{align*}
u_{-\infty}(x) := \lim_{A \downarrow -\infty} u_A(x) =   \lim_{A \downarrow -\infty} \E^x \left[ \int_0^{\tau_A} e^{-rt} f(X_t) \diff
t  + e^{-q \tau_A} g(X_{\tau_A}) 1_{\{ \tau_A < \infty \}}\right].
\end{align*} 
Hence,
\begin{align*}
(u_{A^*})_- (x)  \leq (u_{-\infty})_- (x) \leq \limsup_{A \downarrow -\infty} \E^x \left[ \int_0^{\tau_A} e^{-rt} f_-(X_t) \diff
t  + e^{-q \tau_A} g_-(X_{\tau_A}) 1_{\{ \tau_A < \infty \}}\right] = \E^x \left[ \int_0^{\infty} e^{-rt} f_-(X_t) \diff
t \right],
\end{align*}
where the last equality holds by monotone convergence applied to the $f_-$ term and because $g_-(X_{\tau_A}) \leq -g(A) \vee 0$ on $\{ \tau_A < \infty \}$, which is bounded because $g$ is decreasing.   Because $f$ is increasing, the expectation on the right hand side decreases in $x$ and hence
\begin{align*}
(u_{A^*})_- (x)  \leq \E^{A^*} \left[ \int_0^{\infty} e^{-rt} f_-(X_t) \diff
t \right], \quad x > A^*.
\end{align*}
In sum, $(u_{A^*})_-$ is bounded from above.
%
%
%
%
%
%
%
Recall also Remark \ref{remark_finiteness_of_f}. Hence, Fatou's lemma gives upon $t \uparrow \infty$ and $m \uparrow \infty$
\begin{align*}
\E^x \Big[ e^{-r \tau} u_{A^*}(X_{\tau}) 1_{\{\tau < \infty\}} + \int_0^{\tau} e^{-rs} f(X_s) \diff s \Big] \leq u_{A^*}(x).
\end{align*}
Finally, \eqref{results_egami_yamazaki} shows the result for $A^* \in (-\infty, \infty]$.

(ii) It is now left to show for the case $A^* = -\infty$.  
Because $\partial u_A(x) / \partial A < 0$ for any $A \in \R$ as in \eqref{u_A_derivative}, there again exists $u_{-\infty}(x) := \lim_{A \downarrow -\infty} u_A(x)$.
Assumption \ref{assump_tail} guarantees that $\lim_{A \downarrow -\infty}\E^x [e^{-q \tau_A}|X_{\tau_A}| 1_{\{ \tau_A < \infty \}}] = 0$; see e.g., \cite{Yamazaki_2013}.  
Because the slope of $g_-(x)$ is bounded on the half-line, this shows that 
\begin{align*}\lim_{A \downarrow -\infty}\E^x [e^{-q \tau_A}g(X_{\tau_A}) 1_{\{ \tau_A < \infty \}}] = 0.
\end{align*}
This together with Remark \ref{remark_finiteness_of_f} shows that $u_{-\infty}(x)$ has the desired expression. 

Because $\partial u_A(x) / \partial A < 0$ for any $A \in \R$,  $u_{-\infty}(x) > g(x)$ for any $x \in \R$ (hence the stopping region is an empty set). Moreover, because $u_{-\infty}$ is attained by $\tau^*=\infty$, we have the claim.

\subsection{Proof of Proposition \ref{proposition_general}} \label{proof_proposition_general}
(1,2) We first suppose $x < A^*$.  Then by definition $u_{A^*}(x) = g(x)$.  Because $u_A(x) = g(x)$ for any $A \geq x$, it is sufficient to show $u_A(x) \leq g(x)$ for $A < x$.  This is indeed so because by  \eqref{cond_continuous_fit}, \eqref{u_A_derivative} and $\Lambda(x) < 0$ due to $x < A^*$,
\begin{align*}
u_A(x) \leq u_x(x+) = g(x) +  W^{(r)}(0) \Lambda(x) \leq g(x) = u_{A^*}(x), \quad A < x < A^*.
\end{align*}
This proves (2).
For (1), we additionally show for the case  $x \geq A^*$.  By \eqref{u_A_derivative},  $u_{A^*}(x) \geq u_A(x)$ for any $A \leq x$.  For $A \geq x$, $u_A(x) = g(x)$ and 
 by  \eqref{cond_continuous_fit}, \eqref{u_A_derivative} and $\Lambda(x) > 0$ due to $x > A^*$,
\begin{align*}
u_{A^*}(x) \geq u_x(x+) = g(x) +  W^{(r)}(0) \Lambda(x) \geq g(x) = u_A(x), \quad A \geq x > A^*. 
\end{align*}
Therefore $u_{A^*}(x) \geq u_A(x)$ uniformly in $x \in \R$, as desired.

The corresponding value function (for (1)) can be expressed as the sum of \eqref{Gammas}:
\begin{align*}
\widetilde{u}(x) = g(A^*)  Z^{(r)}(x-A^*) + W^{(r)} (x-A^*)  \left(- \frac r {\Phi_r} g(A^*) + \rho_{g,A^*}^{(r)} + \Psi_f(A^*) \right) - \varphi_{g,A^*}^{(r)}(x)  -
\Theta_f(x; A^*).
\end{align*}
From the definition of $A^*$ that makes \eqref{Large_lambda} vanish,  it is simplified to
\begin{align*}
\widetilde{u}(x) = g(A^*)  Z^{(r)}(x-A^*) + W^{(r)} (x-A^*)  \frac {\sigma^2} 2 g'(A^*) - \varphi_{g,A^*}^{(r)}(x)  -
\Theta_f(x; A^*),
\end{align*}
as desired.

 (3) The proof is the same as that of Proposition \ref{theorem_polynomial}(3).

\bibliographystyle{abbrv}
\bibliographystyle{apalike}

\bibliographystyle{agsm}
\small{\bibliography{ospbib}}

\def\cprime{$'$}
\begin{thebibliography}{10}

\bibitem{alili-kyp}
L.~Alili and A.~E. Kyprianou.
\newblock Some remarks on first passage of {L}\'{e}vy processes, the {A}merican
  put and smooth pasting.
\newblock {\em Ann. Appl. Probab.}, 15:2062--2080, 2004.

\bibitem{Asmussen_2004}
S.~Asmussen, F.~Avram, and M.~R. Pistorius.
\newblock Russian and {A}merican put options under exponential phase-type
  {L}\'evy models.
\newblock {\em Stochastic Process. Appl.}, 109(1):79--111, 2004.

\bibitem{AsmussenMadanPistorius07}
S.~Asmussen, D.~Madan, and M.~R. Pistorius.
\newblock Pricing equity default swaps under an approximation to the {CGMY}
  {L}\'evy model.
\newblock {\em J. Comput. Financ.}, 11(2):79--93, 2007.

\bibitem{avram-et-al-2004}
F.~Avram, A.~E. Kyprianou, and M.~R. Pistorius.
\newblock Exit problems for spectrally negative {L}\'{e}vy processes and
  applications to ({C}anadized) {R}ussion options.
\newblock {\em Ann. Appl. Probab.}, 14:215--235, 2004.

\bibitem{Avram_et_al_2007}
F.~Avram, Z.~Palmowski, and M.~R. Pistorius.
\newblock On the optimal dividend problem for a spectrally negative {L}\'evy
  process.
\newblock {\em Ann. Appl. Probab.}, 17(1):156--180, 2007.

\bibitem{Baurdoux2008}
E.~Baurdoux and A.~E. Kyprianou.
\newblock The {M}c{K}ean stochastic game driven by a spectrally negative
  {L}\'evy process.
\newblock {\em Electron. J. Probab.}, 13:no. 8, 173--197, 2008.

\bibitem{Baurdoux2009}
E.~Baurdoux and A.~E. Kyprianou.
\newblock The {S}hepp-{S}hiryaev stochastic game driven by a spectrally
  negative {L}\'evy process.
\newblock {\em Theory Probab. Appl.}, 53, 2009.

\bibitem{Bertoin_1997}
J.~Bertoin.
\newblock Exponential decay and ergodicity of completely asymmetric {L}\'evy
  processes in a finite interval.
\newblock {\em Ann. Appl. Probab.}, 7(1):156--169, 1997.

\bibitem{Boyarchenko_2007}
S.~Boyarchenko and S.~Levendorski{\u\i}.
\newblock {\em Irreversible decisions under uncertainty}, volume~27 of {\em
  Studies in Economic Theory}.
\newblock Springer, Berlin, 2007.

\bibitem{Boyarchenko_2006}
S.~Boyarchenko and S.~Z. Levendorski{\u\i}.
\newblock General option exercise rules, with applications to embedded options
  and monopolistic expansion.
\newblock {\em Contrib. Theor. Econ.}, 6:Art. 2, 53 pp. (electronic), 2006.

\bibitem{Carmona_dayanik}
R.~Carmona and S.~Dayanik.
\newblock Optimal multiple stopping of linear diffusions.
\newblock {\em Math. Oper. Res.}, 33(2):446--460, 2008.

\bibitem{Touzi_Carmona}
R.~Carmona and N.~Touzi.
\newblock Optimal multiple stopping and valuation of swing options.
\newblock {\em Math. Finance}, 18(2):239--268, 2008.

\bibitem{Chan_2009}
T.~Chan, A.~Kyprianou, and M.~Savov.
\newblock Smoothness of scale functions for spectrally negative {L}\'evy
  processes.
\newblock {\em Probab. Theory Relat. Fields}, 150:691--708, 2011.

\bibitem{Deaton_1992}
A.~Deaton and G.~Laroque.
\newblock On the behaviour of commodity prices.
\newblock {\em Rev. Econ. Stud.}, 59:1--23, 1992.

\bibitem{dixit_pindyck}
A.~Dixit and R.~Pindyck.
\newblock {\em Investment under Uncertainty}.
\newblock Princeton University Press, 1996.

\bibitem{Leung_Yamazaki_2011}
E.~Egami, T.~Leung, and K.~Yamazaki.
\newblock Default swap games driven by spectrally negative {L}\'{e}vy
  processes.
\newblock {\em Stochastic Process. Appl.}, 123(2):347--384, 2013.

\bibitem{Egami-Yamazaki-2011}
M.~Egami and K.~Yamazaki.
\newblock On the continuous and smooth fit principle for optimal stopping
  problems in spectrally negative levy models.
\newblock {\em Adv. in Appl. Probab.}, 46(1).

\bibitem{Egami-Yamazaki-2010-1}
M.~Egami and K.~Yamazaki.
\newblock Precautional measures for credit risk management in jump models.
\newblock {\em Stochastics}, 85(1):111--143, 2013.

\bibitem{Egami_Yamazaki_2010_2}
M.~Egami and K.~Yamazaki.
\newblock Phase-type fitting of scale functions for spectrally negative
  {L}\'evy processes.
\newblock {\em J. Comput. Appl. Math.}, 264:1--22, 2014.

\bibitem{Feldmann_1998}
A.~Feldmann and W.~Whitt.
\newblock Fitting mixtures of exponentials to long-tail distributions to
  analyze network performance models.
\newblock {\em Perform. Evaluation}, (31):245--279, 1998.

\bibitem{Hernandez_Yamazaki_2013}
D.~Hernandez-Hernandez and K.~Yamazaki.
\newblock Games of singular control and stopping driven by spectrally one-sided
  {L}\'evy processes.
\newblock {\em Stochastic Process. Appl.}, forthcoming.

\bibitem{Kuznetsov2013}
A.~Kuznetsov, A.~Kyprianou, and V.~Rivero.
\newblock The theory of scale functions for spectrally negative {L}\'evy
  processes.
\newblock {\em Springer Lecture Notes in Mathematics}, 2061:97--186, 2013.

\bibitem{Kuznetsov_2010}
A.~Kuznetsov, A.~E. Kyprianou, and J.~C. Pardo.
\newblock Meromorphic {L}\'evy processes and their fluctuation identities.
\newblock {\em Ann. Appl. Probab.}, 22:1101--1135, 2012.

\bibitem{Kyprianou_2006}
A.~E. Kyprianou.
\newblock {\em Introductory lectures on fluctuations of {L}\'evy processes with
  applications}.
\newblock Universitext. Springer-Verlag, Berlin, 2006.

\bibitem{Kyprianou_Surya_2007}
A.~E. Kyprianou and B.~A. Surya.
\newblock Principles of smooth and continuous fit in the determination of
  endogenous bankruptcy levels.
\newblock {\em Finance Stoch.}, 11(1):131--152, 2007.

\bibitem{Leland_1994}
H.~E. Leland.
\newblock Corporate debt value, bond covenants, and optimal capital structure.
\newblock {\em J. Finance}, 49(4):1213--1252, 1994.

\bibitem{Leland_Toft_1996}
H.~E. Leland and K.~B. Toft.
\newblock Optimal capital structure, endogenous bankruptcy, and the term
  structure of credit spreads.
\newblock {\em J. Finance}, 51(3):987--1019, 1996.

\bibitem{Leung_Yamazaki_2010}
T.~Leung and K.~Yamazaki.
\newblock American step-up and step-down default swaps under {L}\'{e}vy models.
\newblock {\em Quant. Finance}, 13(1):137--157, 2013.

\bibitem{mordecki}
E.~Mordecki.
\newblock Optimal stopping and perpetual options for {L}\'{e}vy processes.
\newblock {\em Finance Stoch.}, 6:473--493, 2002.

\bibitem{sulem}
B.~{\O}ksendal and A.~Sulem.
\newblock {\em Applied Stochastic Control of Jump Diffusions}.
\newblock Springer, New York, 2005.

\bibitem{peskir-shiryaev}
G.~Peskir and A.~N. Shiryaev.
\newblock {\em Optimal stopping and Free-Boundary Problems (Lectures in
  Mathematics, ETH Z\"{u}rich)}.
\newblock Birkhauser, Basel, 2006.

\bibitem{ProtterBook}
P.~Protter.
\newblock {\em Stochastic integration and differential equations}.
\newblock Springer, 2005.

\bibitem{Surya_2008}
B.~A. Surya.
\newblock Evaluating scale functions of spectrally negative {L}\'evy processes.
\newblock {\em J. Appl. Probab.}, 45(1):135--149, 2008.

\bibitem{Yamazaki_2013}
K.~Yamazaki.
\newblock Inventory control for spectrally positive {L}\'evy demand processes.
\newblock {\em arXiv:1303.5163}, 2013.

\bibitem{Yang_Brorsen}
S.-R. Yang and B.~W. Brorsen.
\newblock Nonlinear dynamics of daily cash prices.
\newblock {\em Amer. J. Agr. Econ.}, 74(3):706--715, 1992.

\end{thebibliography}

\end{document}